\documentclass{article}

\usepackage[margin=2.5cm]{geometry}
\usepackage{amsmath,amssymb,amsthm,mathtools}
\usepackage{mathrsfs}
\usepackage{hyperref}
\numberwithin{equation}{section}
\newtheorem{theorem}{Theorem}[section]
\newtheorem{lemma}[theorem]{Lemma}
\newtheorem{proposition}[theorem]{Proposition}
\newtheorem{corollary}[theorem]{Corollary}
\theoremstyle{definition}
\newtheorem{definition}[theorem]{Definition}
\theoremstyle{remark}
\newtheorem{remark}[theorem]{Remark}

\DeclareMathOperator{\dist}{dist}
\DeclareMathOperator{\Adm}{Adm}
\DeclareMathOperator{\Lip}{Lip}
\DeclareMathOperator{\supp}{supp}
\providecommand{\restriction}{\mathbin{\upharpoonright}}

\begin{document}

\title{Sharp Continuity Moduli for Dirichlet Heat Flow in Boundary-Reservoir Transport Metrics}

\author{Maja Gw\'o\'zd\'z\\ ETH Z\"{u}rich\\\texttt{mgwozdz@ethz.ch}}
\date{}

\maketitle

\begin{abstract}
Let \(\Omega\subset\mathbb R^n\) be a bounded \(C^2\) open set and let
\(P_t\) be the killed Dirichlet heat semigroup. We prove the sharp fixed-time power-scale modulus of \(P_t\) for the Figalli--Gigli boundary-reservoir transport distances \(W_{b,p}\). For every \(t>0\), \(P_t\) is globally Lipschitz with respect to \(W_{b,1}\). For every \(p>1\), and on every total-mass sublevel \(\{\mu:\mu(\Omega)\le m\}\), it is \(1/p\)-H\"older:
\[
W_{b,p}(P_t\mu,P_t\nu)^p
\le C_{t,p,m,\Omega} W_{b,p}(\mu,\nu).
\]
For \(p>1\), we show that the exponent \(1/p\) is optimal in the scale of power moduli. On the full finite-measure space, \(P_t\) is discontinuous at the zero measure. To establish the lower bound, we rely on the amplification of the boundary layer. More precisely, a unit mass initially placed at distance \(\varepsilon\) from \(\partial\Omega\) has input \(W_{b,p}\)-distance \(O(\varepsilon)\) from zero, whereas after any fixed positive time, its \(p\)-th boundary moment is bounded below by \(c\varepsilon\). As a result, in the quadratic case and in the original finite-measure \(W_{b,2}\) metric, there does not exist a standard finite-\(\lambda\)
\(\mathrm{EVI}_\lambda\) semigroup on a \(W_{b,2}\)-metric domain which would contain the affine constant-boundary data class and could restrict to the affine
constant-boundary Dirichlet heat flow. Finally, we also describe the corresponding lower-bound obstruction for smooth uniformly elliptic perturbations in divergence form.
\end{abstract}



\section{Introduction}
Figalli and Gigli \cite{FG10} introduced the boundary-reservoir distance for Dirichlet heat equations in optimal transport. In this geometry, unmatched mass may be
transported to or from the boundary, so, intuitively, the boundary acts as an infinite
reservoir. In the quadratic case, the entropy minimising-movement scheme yields the heat flow with a constant Dirichlet boundary value \cite[Theorem~3.5]{FG10}. 
Several variants of this idea have been developed in further works. These include, among others, the double-space or charged formulations in \cite{PS20}, transport-cost alternatives \cite{Mor18}, the dynamic representations of the Benamou--Brenier type \cite{EM}, the modified Jordan--Kinderlehrer--Otto (JKO) schemes for Fokker--Planck equations with Dirichlet boundary data \cite{Qua26}, and the porous-medium minimising-movement schemes \cite{KKS}.

In this paper, we describe a fixed-time obstruction inside the boundary-extension metric. We ask whether, for fixed \(t>0\), the killed Dirichlet heat semigroup \(P_t\) acts regularly on \((\mathcal M(\Omega),W_{b,p})\). Our answer is sharp. Indeed, at \(p=1\), we show that \(P_t\) is globally Lipschitz. At every \(p>1\), the mass-bounded modulus is not Lipschitz but exactly \(1/p\)-H\"older. Moreover, the map is discontinuous at the zero measure on the full finite-measure space. For \(m>0\), we write
\[
\mathcal M_{\le m}(\Omega)
:=\{\mu\in\mathcal M(\Omega):\mu(\Omega)\le m\}.
\]
In the introduction, \(P_t\) on finite measures denotes the killed Dirichlet heat semigroup that acts by duality on \(C_0(\Omega)\). We also set
\[
C_{\partial\Omega}(\overline\Omega)
:=\{f\in C(\overline\Omega):f|_{\partial\Omega}=0\}.
\]
Therefore,
\[
\int_\Omega \varphi\,d(P_t\mu)
=
\int_\Omega P_t\varphi\,d\mu,
\qquad
\varphi\in C_{\partial\Omega}(\overline\Omega).
\]
We recall the formal definition in Definition~\ref{def:Pt-on-measures}. For the Laplacian with fixed constant boundary value \(c\), we write \(S_t\) when \(c\) is fixed in the context and \(S_t^c\) when we want to emphasise the dependence on \(c\). For the elliptic operator \(\mathcal L\), we denote the analogous affine map by \(S_t^{\mathcal L}\).

\begin{theorem}[Sharp fixed-time modulus for the killed Dirichlet heat flow]
\label{thm:sharp-modulus}
Let \(\Omega\subset\mathbb R^n\) be a nonempty bounded \(C^2\) open set and let
\(P_t\) be the killed Dirichlet heat semigroup. Let \(t>0\) be fixed.
\begin{enumerate}
\item There exists \(L_{t,\Omega}<\infty\) such that
\[
W_{b,1}(P_t\mu,P_t\nu)
\le L_{t,\Omega} W_{b,1}(\mu,\nu)
\qquad\forall \mu,\nu\in\mathcal M(\Omega).
\]
\item Let \(p>1\) and \(m>0\). In this case, there exists \(C_{t,p,m,\Omega}<\infty\) such that
\[
W_{b,p}(P_t\mu,P_t\nu)^p
\le
C_{t,p,m,\Omega} W_{b,p}(\mu,\nu)
\qquad
\forall \mu,\nu\in\mathcal M_{\le m}(\Omega).
\]
In other terms, \(P_t\) is \(1/p\)-H\"older on every total-mass sublevel:
\[
W_{b,p}(P_t\mu,P_t\nu)
\le
C_{t,p,m,\Omega}^{1/p}
W_{b,p}(\mu,\nu)^{1/p}.
\]
\item For every \(p>1\), the H\"older exponent \(1/p\) is sharp in the scale
of power moduli. More precisely, for every \(\alpha>1/p\) and every \(m>0\),
\[
\sup_{\substack{\mu,\nu\in\mathcal M_{\le m}(\Omega)\\ \mu\ne\nu}}
\frac{W_{b,p}(P_t\mu,P_t\nu)}
     {W_{b,p}(\mu,\nu)^\alpha}
=+\infty .
\]
In particular, \(P_t\) has infinite \(W_{b,p}\)-Lipschitz constant on \(\mathcal M_{\le m}(\Omega)\).
\item On the full finite-measure space \((\mathcal M(\Omega),W_{b,p})\), the map
\(P_t\) is discontinuous at \(0\) for every \(p>1\).
\end{enumerate}
\end{theorem}

The exponent \(1/p\) already appears in the one-dimensional half-line model. For the killed heat kernel on \((0,\infty)\), a unit packet initially placed at height \(\varepsilon\) has initial boundary \(p\)-moment of order \(\varepsilon^p\). At a fixed positive time, the Dirichlet kernel is linear in the starting height near the boundary. In other terms, the adjoint solution has a Hopf-type linear lower bound. It follows that the positive-time boundary \(p\)-moment is of order \(\varepsilon\), and the \(W_{b,p}\)-distance to zero is of order \(\varepsilon^{1/p}\). This explains the idea behind the sharp H\"older exponent.

The discontinuity at zero relies on sequences whose total masses diverge while their boundary \(p\)-moments tend to zero. On fixed total-mass sublevels, the map is continuous, and, in fact, it is \(1/p\)-H\"older by Theorem~\ref{thm:sharp-modulus}. This implies that the pathology is not simply due to lack of smoothing. Indeed, it shows the incompatibility between the boundary reservoir topology and the unbounded amounts of mass that concentrate in the vanishing boundary layers.

\subsection{Related works}
Figalli and Gigli \cite{FG10} introduced \(W_{b,2}\) as a boundary-reservoir transport
geometry for nonnegative measures. They applied a minimising-movement scheme to obtain the heat equation with constant Dirichlet value \(1\) for the entropy
\(\int_\Omega(\rho\log\rho-\rho)\,dx\). Ambrosio--Gigli later posed the related \(W_{b,2}\) heat-flow question for the entropy \(\int_\Omega\rho\log\rho\,dx\), where the reservoir equilibrium value is \(e^{-1}\) \cite[Theorem~6.6 and Open Problem~6.7]{AG13}.

Other variational Dirichlet optimal-transport structures include the transport-cost framework of Morales \cite{Mor18}, the dynamic Benamou--Brenier and curves-of-maximal-slope formulation of Erbar--Meglioli \cite{EM}, the \(W_{b,2}\)-based porous
medium construction of Kim--Koo--Seo \cite{KKS}, and the modified JKO/Fokker--Planck framework of Quattrocchi \cite{Qua26}. The closest recent variational results to ours are \cite{EM, Qua26}. Erbar--Meglioli \cite{EM} discuss the dynamic formulation of the Figalli--Gigli-type boundary-reservoir distance and a curves-of-maximal-slope statement for nonlinear diffusion with constant nonnegative Dirichlet data. Additionally, Quattrocchi \cite{Qua26} establishes the convergence of a modified JKO scheme for general strictly positive Dirichlet data, which are independent of time. Moreover, in one dimension, \cite{Qua26} gives a curve-of-maximal-slope interpretation for a modified metric. Our results are orthogonal to those constructions. To reiterate, we prove the sharp fixed-time power-scale semigroup modulus, that is, the optimal H\"{o}lder exponent in the power-modulus scale, and thus establish a direct positive-time Lipschitz/EVI obstruction in the static finite-measure metric of the \(W''_p\) type.

Other representations, such as the static shortcut, sticky-reflecting diffusion \cite{CMS25}, relative transport, unbalanced optimal transport (\textit{cf}. \cite{BGP24, Che, BE24}), the relaxation framework of Savar\'e--Sodini \cite{SS24}, help understand the boundary reservoir geometry but do not address the fixed-time semigroup modulus. In our proofs, we also rely on the standard relative-geometry results, especially, the shortcut/Kantorovich--Rubinstein representation, which we apply in the \(p=1\) part below. The new point here is that we combine it with the fixed-time Dirichlet regularisation and the matching \(p>1\) boundary-layer lower bound in the finite-measure boundary-reservoir metric. Another recent non-existence result is given in \cite{BMRR}, but the problem they consider is Wentzell, not Dirichlet, and their obstruction focuses on Wasserstein/JKO representations.

To the best of our knowledge, the sharp fixed-time power-scale
\(W_{b,p}\)-modulus of the killed Dirichlet semigroup in the finite-measure Figalli--Gigli \(W''_p\)-type metric has not been previously studied. Our contribution focuses on the identification of the exact fixed-time modulus in the static finite-measure metric, that is, Lipschitz continuity at \(p=1\), sharp \(1/p\)-H\"older continuity on mass sublevels for every \(p>1\), and discontinuity at the zero measure on the full space with finite measure.

\subsection{Organisation}

We divide the proof into metric and PDE parts. In the metric part, we apply the elementary comparison between \(W_{b,1}\) and \(W_{b,p}\) on
mass sublevels. The PDE part focuses on the Hopf-type amplification estimate for data in boundary layers. For the boundary-layer lower bound, we use positivity, duality, and a linear Hopf-type lower bound for the adjoint flow. This means that we may then extend the sharpness method from the lower-bound side of \(p>1\) to smooth uniformly elliptic operators in divergence form with first-order drift and without a zero-order term. In particular, we extend the statements on full-space discontinuity, the infinite-Lipschitz, and obstruction w.r.t. the affine constant-boundary to the density-induced maps.

In Section~\ref{sec:prelim}, we briefly recall the boundary-reservoir metric and the elementary metric estimates needed for the proofs. In Section~\ref{sec:Wbp-ell}, we establish the boundary-layer amplification estimate for the Dirichlet Laplacian.
Section~\ref{sec:Wb1} proves the \(W_{b,1}\) endpoint estimate and completes the
proof of Theorem~\ref{thm:sharp-modulus}. In Section~\ref{sec:ell-robustness},
we extend the uniformly elliptic lower bound. Section~\ref{sec:fg-evi}
derives the results for the quadratic Figalli--Gigli metric, which includes the
failure of the positive-time Lipschitz/EVI behaviour for the affine Dirichlet heat flow with a constant boundary.

For \(p=2\), we show the consequences for the Figalli--Gigli metric. On
an arbitrary \(W_{b,2}\)-metric domain that contains \(\mathcal X_c\), the affine
Dirichlet heat flow cannot be the restriction of a standard finite-\(\lambda\) \(\mathrm{EVI}_\lambda\) semigroup. The precise statement is given in Section~\ref{sec:fg-evi}. We recall only the elementary facts needed for the proof construction, namely, the distance to the zero measure, homogeneity, and the two comparison inequalities \(W_{b,p}^p\lesssim W_{b,1}\) and \(W_{b,1}\lesssim_m W_{b,p}\). The other metric properties of \(W_{b,p}\) are standard in the metric-pair framework \cite{FG10,Che}, and we recall them in Appendix~\ref{app:Wbp-metric-proof} to fix normalisation.

\section{Boundary transport distances and boundary moments}\label{sec:prelim}

\subsection{Notation and assumptions}
\(\Omega\subset\mathbb R^n\) is a nonempty bounded open set with \(C^2\) boundary, where we do not assume connectedness. We use the term \(C^2\) boundary in the standard one-sided sense, that is, for every \(\xi\in\partial\Omega\), there exists a neighbourhood \(U\) of \(\xi\), a rigid motion, and a \(C^2\) function \(h:\mathbb R^{n-1}\to\mathbb R\) such that, in these coordinates,
\[
U\cap\Omega=\{(x',x_n)\in U:\ x_n>h(x')\},
\qquad
U\cap\partial\Omega=\{(x',x_n)\in U:\ x_n=h(x')\}.
\]
We set
\[
\delta(x):=\dist(x,\partial\Omega)\qquad (x\in\Omega),
\]
extended by $\delta=0$ on $\partial\Omega$. We also set
\[
\mathcal M(\Omega):=\{\mu\ge0:\ \mu \text{ is a finite Radon measure on }\Omega\}.
\]
For \(m>0\),
\[
\mathcal M_{\le m}(\Omega)
:=\{\mu\in\mathcal M(\Omega):\mu(\Omega)\le m\}.
\]
For \(p\in[1,\infty)\) and \(\mu\in\mathcal M(\Omega)\), we define the \(p\)-th boundary moment by
\[
M_p(\mu):=\int_\Omega \delta(x)^p\,d\mu(x).
\]
We work only with finite Radon measures on \(\Omega\). Since \(\Omega\) is bounded, \(\delta\) is also bounded, and
\[
\int_\Omega \delta(x)^p\,d\mu(x)<\infty
\qquad\forall \mu\in\mathcal M(\Omega),\ \forall p\in[1,\infty).
\]
All boundary moments below are finite. We write \(0\) for the zero measure on \(\Omega\) and
\[
D_\Omega:=\sup\{|x-y|:\ x,y\in\overline\Omega\}.
\]
We use the following elementary fact that stems from the \(C^2\) boundary assumption.

\begin{lemma}\label{lem:components}
If \(\Omega\subset\mathbb R^n\) is a nonempty bounded open set with \(C^2\) boundary, then \(\Omega\) has finitely many connected components \(\Omega_1,\dots,\Omega_N\). Each \(\Omega_j\) is a bounded \(C^2\) domain, the closures \(\overline{\Omega_1},\dots,\overline{\Omega_N}\) are pairwise disjoint, and for every \(x\in\Omega_j\),
\[
\dist(x,\partial\Omega_j)=\dist(x,\partial\Omega)=\delta(x).
\]
\end{lemma}

\begin{proof}
See Appendix~\ref{app:components}.
\end{proof}

\subsection{Boundary transport distances \texorpdfstring{$W_{b,p}$}{Wb,p}}
We write \(\pi_1,\pi_2:\overline\Omega\times\overline\Omega\to\overline\Omega\)
for the coordinate projections. If \(\Gamma\) is a finite Borel measure on a product
space, we write \(\Gamma_1,\Gamma_2\) for its first and second marginals when there is no ambiguity.

\begin{definition}[Admissible plans and $W_{b,p}$]\label{def:Wbp}
We fix \(p\in[1,\infty)\) and let \(\mu,\nu\in\mathcal M(\Omega)\). We call a finite nonnegative Borel measure \(\gamma\) on \(\overline\Omega\times\overline\Omega\) \textit{admissible} for \((\mu,\nu)\) if
for every Borel set \(E\subset\Omega\),
\[
(\pi_1)_\#\gamma(E)=\mu(E),
\qquad
(\pi_2)_\#\gamma(E)=\nu(E).
\]
We write \(\Adm(\mu,\nu)\) for the set of all such admissible plans, and define
\[
W_{b,p}(\mu,\nu)^p:=\inf_{\gamma\in\Adm(\mu,\nu)}\int_{\overline\Omega\times\overline\Omega}|x-y|^p\,d\gamma(x,y).
\]
\end{definition}

\begin{lemma}\label{lem:adm-nonempty}
We fix $p\in[1,\infty)$ and $\mu,\nu\in\mathcal M(\Omega)$. It follows that $\Adm(\mu,\nu)\neq\emptyset$ and $W_{b,p}(\mu,\nu)<\infty$. More precisely, for an arbitrary fixed $b_0\in\partial\Omega$, we have
\[
W_{b,p}(\mu,\nu)^p\le D_\Omega^p\bigl(\mu(\Omega)+\nu(\Omega)\bigr).
\]
\end{lemma}

\begin{proof}
We fix $b_0\in\partial\Omega$ and define
\[
\gamma:=(\mathrm{Id},b_0)_\#\mu + (b_0,\mathrm{Id})_\#\nu.
\]
For every Borel set $E\subset\Omega$,
\[
(\pi_1)_\#\gamma(E)=\mu(E)+\nu(\Omega)\,\mathbf 1_E(b_0)=\mu(E),
\]
because $b_0\in\partial\Omega$ and $E\subset\Omega$. Analogously, we have
\[
(\pi_2)_\#\gamma(E)=\nu(E)+\mu(\Omega)\,\mathbf 1_E(b_0)=\nu(E).
\]
This implies that $\gamma\in\Adm(\mu,\nu)$, so $\Adm(\mu,\nu)\neq\emptyset$. Moreover,
\[
\begin{multlined}
\int_{\overline\Omega\times\overline\Omega}|x-y|^p\,d\gamma
=\int_\Omega |x-b_0|^p\,d\mu(x)\\
+\int_\Omega |b_0-y|^p\,d\nu(y)
\le D_\Omega^p\bigl(\mu(\Omega)+\nu(\Omega)\bigr)<\infty.
\end{multlined}
\]
It now suffices to take the infimum over admissible plans to prove finiteness and the bound.
\end{proof}

\begin{remark}[Boundary--boundary mass is irrelevant]\label{rem:discard-bdrybdry}
We apply the following observation multiple times in the proofs. If \(\gamma\in\Adm(\mu,\nu)\), then the deletion of its restriction to \(\partial\Omega\times\partial\Omega\) does not change the interior marginals and can only decrease the transport cost.
\end{remark}

\begin{lemma}\label{lem:approx-boundary-selection}
For every \(m\in\mathbb N\), there exists a Borel map \(\beta_m:\Omega\to\partial\Omega\) such that
\[
|x-\beta_m(x)|\le \Bigl(1+\frac1m\Bigr)\delta(x)
\qquad\forall x\in\Omega.
\]
\end{lemma}

\begin{proof}
Let us fix a countable dense subset \(\{b_k\}_{k\ge1}\subset\partial\Omega\). For \(x\in\Omega\), we set
\[
I_m(x):=\Bigl\{k\in\mathbb N:\ |x-b_k|\le \Bigl(1+\frac1m\Bigr)\delta(x)\Bigr\}.
\]
Notice that each \(I_m(x)\) is nonempty. We choose \(b\in\partial\Omega\) with \(|x-b|=\delta(x)\), and then choose \(k\) such that \(|b-b_k|\le \delta(x)/m\). Let us also define
\[
k_m(x):=\min I_m(x),
\qquad
\beta_m(x):=b_{k_m(x)}.
\]
For each \(k\in\mathbb N\), it holds that
\[
\{x\in\Omega:\ k_m(x)=k\}
=
\{x:\ k\in I_m(x)\}\cap\bigcap_{j<k}\{x:\ j\notin I_m(x)\},
\]
so \(k_m\) and \(\beta_m\) are Borel. By construction, we obtain
\[
|x-\beta_m(x)|\le \Bigl(1+\frac1m\Bigr)\delta(x)
\qquad\forall x\in\Omega. \qedhere
\]
\end{proof}

\begin{lemma}[Distance to the zero measure]\label{lem:tozero}
We fix $p\in[1,\infty)$. For every $\mu\in\mathcal M(\Omega)$, we have
\[
W_{b,p}(\mu,0)^p=\int_\Omega \delta(x)^p\,d\mu(x).
\]
\end{lemma}

\begin{proof}
Let \(\gamma\in\Adm(\mu,0)\). Since \((\pi_2)_\#\gamma(E)=0\) for every
Borel set \(E\subset\Omega\), we have
\[
\gamma(\overline\Omega\times\Omega)=0.
\]
This means that \(y\in\partial\Omega\) for \(\gamma\)-a.e.\ \((x,y)\), and
\(|x-y|\ge \delta(x)\) for \(\gamma\)-a.e.\ \((x,y)\). Because \(\delta=0\) on
\(\partial\Omega\) and \((\pi_1)_\#\gamma\restriction_{\Omega}=\mu\), it follows that
\[
\int_{\overline\Omega\times\overline\Omega}|x-y|^p\,d\gamma
\ge
\int_{\overline\Omega}\delta(x)^p\,d(\pi_1)_\#\gamma(x)
=
\int_\Omega \delta(x)^p\,d\mu(x).
\]
We now take the infimum over \(\gamma\) to obtain
\[
W_{b,p}(\mu,0)^p\ge \int_\Omega \delta(x)^p\,d\mu(x).
\]

For \(m\in\mathbb N\), let \(\beta_m:\Omega\to\partial\Omega\) be the Borel map from Lemma~\ref{lem:approx-boundary-selection}, and set
\[
\gamma_m:=(\mathrm{Id},\beta_m)_\#\mu.
\]
It follows that \(\gamma_m\in\Adm(\mu,0)\), and
\[
\begin{aligned}
\int |x-y|^p\,d\gamma_m(x,y)
&=\int_\Omega |x-\beta_m(x)|^p\,d\mu(x)\\
&\le \Bigl(1+\frac1m\Bigr)^p\int_\Omega \delta(x)^p\,d\mu(x).
\end{aligned}
\]
It now suffices to take the infimum over admissible plans and let \(m\to\infty\) to arrive at the reverse inequality.
\end{proof}

\begin{remark}\label{rem:Wbp-homogeneity}
We fix \(p\in[1,\infty)\). For every \(a\ge0\) and every \(\mu,\nu\in\mathcal M(\Omega)\),
\[
W_{b,p}(a\mu,a\nu)=a^{1/p}W_{b,p}(\mu,\nu).
\]
Indeed, if \(a>0\), then \(\gamma\in\Adm(\mu,\nu)\) if and only if
\(a\gamma\in\Adm(a\mu,a\nu)\), and
\[
\int_{\overline\Omega\times\overline\Omega}|x-y|^p\,d(a\gamma)
=
a\int_{\overline\Omega\times\overline\Omega}|x-y|^p\,d\gamma.
\]
Taking infima yields
\[
W_{b,p}(a\mu,a\nu)^p=a\,W_{b,p}(\mu,\nu)^p.
\]
The case \(a=0\) is trivial.
\end{remark}

\begin{lemma}[Comparison of \(W_{b,1}\) and \(W_{b,p}\)]
\label{lem:Wb1-Wbp-comparison}
Let \(p>1\). The following properties hold.

\begin{enumerate}
\item For all \(\mu,\nu\in\mathcal M(\Omega)\),
\[
W_{b,p}(\mu,\nu)^p
\le D_\Omega^{p-1} W_{b,1}(\mu,\nu).
\]

\item If \(\mu,\nu\in\mathcal M_{\le m}(\Omega)\), then
\[
W_{b,1}(\mu,\nu)
\le (2m)^{1-1/p} W_{b,p}(\mu,\nu).
\]
\end{enumerate}
\end{lemma}

\begin{proof}
For the first point, let \(\gamma\in\Adm(\mu,\nu)\). Since \(|x-y|\le D_\Omega\) on \(\overline\Omega\times\overline\Omega\),
\[
\int |x-y|^p\,d\gamma
\le D_\Omega^{p-1}\int |x-y|\,d\gamma.
\]
It suffices to take the infimum over \(\gamma\) to obtain the claim.

For the second point, we fix \(\gamma\in\Adm(\mu,\nu)\) and delete its \(\partial\Omega\times\partial\Omega\)-part. The new admissible plan, which we denote by \(\gamma\), has total mass at most \(\mu(\Omega)+\nu(\Omega)\le 2m\). By H\"older's inequality, we obtain
\[
\int |x-y|\,d\gamma
\le
\gamma(\overline\Omega\times\overline\Omega)^{1-1/p}
\left(\int |x-y|^p\,d\gamma\right)^{1/p}
\le
(2m)^{1-1/p}
\left(\int |x-y|^p\,d\gamma\right)^{1/p}.
\]
Taking the infimum over admissible \(\gamma\) yields the claim.
\end{proof}

For \(p\in[1,\infty)\), \(X\subset\mathcal M(\Omega)\), and \(F:X\to\mathcal M(\Omega)\), we write
\[
\Lip_{W_{b,p}}(F;X)
:=\sup_{\substack{\mu,\nu\in X\\ \mu\neq \nu}}
\frac{W_{b,p}(F\mu,F\nu)}{W_{b,p}(\mu,\nu)}
\in[0,\infty]
\]
for the global \(W_{b,p}\)-Lipschitz constant of \(F\) on \(X\). When \(X=\mathcal M(\Omega)\), we simply write \(\Lip_{W_{b,p}}(F)\).

\begin{lemma}[Basic metric properties of \texorpdfstring{$W_{b,p}$}{Wb,p}]\label{lem:Wbp-metric}
We fix $p\in[1,\infty)$. For all $\mu,\nu,\sigma\in\mathcal M(\Omega)$, the following properties hold.
\begin{itemize}
\item $W_{b,p}(\mu,\nu)=W_{b,p}(\nu,\mu)$
\item $W_{b,p}(\mu,\mu)=0$, and $W_{b,p}(\mu,\nu)=0$ implies $\mu=\nu$ as measures on $\Omega$
\item $W_{b,p}(\mu,\sigma)\le W_{b,p}(\mu,\nu)+W_{b,p}(\nu,\sigma)$
\end{itemize}
It follows that $W_{b,p}$ is a metric on $\mathcal M(\Omega)$.
\end{lemma}

\begin{remark}
For \(p=2\), this is precisely the Figalli--Gigli distance \cite{FG10}. For general \(p\),
this is the specialisation to \((\overline\Omega,\partial\Omega)\) of Che's
metric-pair partial-transport distance (see \cite[Section~3, Theorems~3.7 and~3.10]{Che}). This framework establishes the existence of optimal partial-transport plans and the metric property. Since \(\Omega\) is bounded, the \(p\)-finiteness condition holds automatically for finite Radon measures on \(\Omega\). We will use Lemma~\ref{lem:Wbp-metric} multiple times, and defer its proof to Appendix~\ref{app:Wbp-metric-proof}.
\end{remark}

For a finite signed Radon measure \(\sigma\) on \(\Omega\), we write
\[
\|\sigma\|_{\mathrm{TV}}:=|\sigma|(\Omega).
\]
In particular, if \(\sigma=h\,dx\) with \(h\in L^1(\Omega)\), then
\[
\|\sigma\|_{\mathrm{TV}}=\|h\|_{L^1(\Omega)}.
\]

\section{Dirichlet Laplacian and boundary amplification}\label{sec:Wbp-ell}
The proof of the \(p>1\) instability has two steps. We first construct boundary-layer data with unit mass with support at distance \(\asymp\varepsilon\) from
\(\partial\Omega\). We then apply a Hopf-type lower bound for the adjoint Dirichlet flow so that, after any fixed positive time, a boundary-moment lower bound of order
\(\varepsilon\) follows. We establish the robustness of the same lower-bound idea for more general uniformly elliptic operators in Section~\ref{sec:ell-robustness}.

\subsection{Geometry of \texorpdfstring{$\delta$}{delta} near \texorpdfstring{$\partial\Omega$}{the boundary}}

\begin{lemma}[Tubular neighbourhood and signed distance]\label{lem:tubular}
We assume that \(\Omega\subset\mathbb R^n\) is a nonempty bounded open set with a \(C^2\) boundary. In this case, there exist \(r_0>0\), an open tubular neighbourhood
\[
U:=\{x\in\mathbb R^n:\dist(x,\partial\Omega)<r_0\},
\]
and a function \(\tilde\delta\in C^2(U)\) such that the normal map
\[
\Phi:\partial\Omega\times(-r_0,r_0)\to U,
\qquad
\Phi(\xi,s):=\xi+s\,\mathbf n(\xi),
\]
is a \(C^1\)-diffeomorphism onto \(U\). Moreover,
\[
\Pi(\Phi(\xi,s))=\xi,\qquad
\tilde\delta(\Phi(\xi,s))=s,
\]
where \(\Pi:U\to\partial\Omega\) is the nearest-point projection and
\(\mathbf n\) is the inward unit normal. In particular, the following statements hold.
\begin{itemize}
\item \(\tilde\delta=0\) on \(\partial\Omega\), \(\tilde\delta>0\) in
\(\Omega\cap U\), and \(\tilde\delta<0\) on \(U\setminus\overline\Omega\);
\item \(|\nabla\tilde\delta|=1\) on \(U\);
\item each \(x\in U\) has the unique representation
\[
x=\Pi(x)+\tilde\delta(x)\mathbf n(\Pi(x)).
\]
\end{itemize}
It follows that \(\delta=\tilde\delta\) on \(\Omega\cap U\), so \(\delta\in C^2(\{x\in\overline\Omega:\delta(x)<r_0\})\), \(|\nabla\delta|=1\) on \(\{0<\delta<r_0\}\), and \(\Delta\delta\) is bounded on \(\{x\in\Omega:\delta(x)\le r_0/2\}\).
\end{lemma}

\begin{proof}
The existence of a tubular neighbourhood with unique nearest-point projection
and normal-coordinate representation is standard for compact \(C^2\) hypersurfaces (consult, for instance, the discussion on \cite[pp.~153--154]{Foo84}). The \(C^2\)-regularity of the signed distance near a compact \(C^2\) hypersurface is also classical (see \cite{KP81} and \cite{Foo84}). The boundedness of \(\Delta\delta\) on
\(\{x\in\Omega:\delta(x)\le r_0/2\}\) then follows immediately from the \(C^2\)-regularity of \(\delta\) on \(\{x\in\overline\Omega:\delta(x)<r_0\}\).
\end{proof}

\begin{lemma}[Smooth unit-mass collar data near a given boundary point]\label{lem:collar-bump}
Let \(\Omega\subset\mathbb R^n\) be a nonempty bounded open set with \(C^2\)
boundary. There exists $\varepsilon_0>0$ such that for every choice of $\xi_0\in\partial\Omega$ and every $\varepsilon\in(0,\varepsilon_0)$ there is a function $u_0^\varepsilon\in C_c^\infty(\Omega)$, $u_0^\varepsilon\ge 0$, with
\[
\begin{aligned}
\supp(u_0^\varepsilon)
&\subset\{x\in\Omega:\ \varepsilon\le \delta(x)\le 2\varepsilon\},\\
\int_{\Omega}u_0^\varepsilon\,dx&=1.
\end{aligned}
\]
Moreover, we may choose $u_0^\varepsilon$ with
\[
\supp(u_0^\varepsilon)\subset B_{\varepsilon/4}\!\left(\xi_0+\frac{3\varepsilon}{2}\,\mathbf n(\xi_0)\right).
\]
\end{lemma}

\begin{proof}
Let $r_0$ be from Lemma~\ref{lem:tubular} and set $\varepsilon_0:=r_0/4$. We fix $\xi_0\in\partial\Omega$ and $\varepsilon\in(0,\varepsilon_0)$, and define
\[
x_\varepsilon:=\xi_0+\frac{3\varepsilon}{2}\,\mathbf n(\xi_0).
\]
It follows that $\delta(x_\varepsilon)=3\varepsilon/2$. We choose $\chi\in C_c^\infty(B_{1/4}(0))$ with $\chi\ge 0$ and $\int_{\mathbb R^n}\chi\,dx=1$, and define
\[
u_0^\varepsilon(x):=\varepsilon^{-n}\chi\!\left(\frac{x-x_\varepsilon}{\varepsilon}\right).
\]
It follows that $u_0^\varepsilon\in C_c^\infty(\mathbb R^n)$, $u_0^\varepsilon\ge0$, $\int_{\mathbb R^n}u_0^\varepsilon\,dx=1$, and $\supp(u_0^\varepsilon)\subset B_{\varepsilon/4}(x_\varepsilon)$. Since $x\mapsto \dist(x,\partial\Omega)$ is $1$-Lipschitz, for every $x\in B_{\varepsilon/4}(x_\varepsilon)$,
\[
\frac{5\varepsilon}{4}
\le \dist(x,\partial\Omega)
\le \frac{7\varepsilon}{4}.
\]
Since $\dist(x_\varepsilon,\partial\Omega)=3\varepsilon/2>\varepsilon/4$, we have $B_{\varepsilon/4}(x_\varepsilon)\cap\partial\Omega=\varnothing$. Suppose there existed $y\in B_{\varepsilon/4}(x_\varepsilon)\setminus\Omega$, then the segment $[x_\varepsilon,y]$ would intersect $\partial\Omega$ at some point $z$. In particular $|x_\varepsilon-z|\le |x_\varepsilon-y|<\varepsilon/4$, which would contradict $\dist(x_\varepsilon,\partial\Omega)=3\varepsilon/2$. It must hold that $B_{\varepsilon/4}(x_\varepsilon)\subset\Omega$. We then obtain
\[
\supp(u_0^\varepsilon)\subset\{x\in\Omega:\ \varepsilon\le \delta(x)\le 2\varepsilon\}.
\]
Finally, we conclude that $u_0^\varepsilon\in C_c^\infty(\Omega)$, $u_0^\varepsilon\ge0$, and $\int_\Omega u_0^\varepsilon\,dx=1$.
\end{proof}

\subsection{Dirichlet semigroups}\label{subsec:dirichlet-semigroups}
For \(c\ge0\), we set
\[
\mathcal D_c
:=
\bigl\{\rho\in C^\infty(\overline\Omega): \rho\ge0 \text{ on }\overline\Omega,\ \rho\equiv c \text{ near }\partial\Omega\bigr\},
\]
and identify \(\rho\) with \(\rho\,dx\) whenever it is convenient. We also set
\[
\mathcal X_c:=\{\rho\,dx:\rho\in\mathcal D_c\}.
\]
We use \(\mathcal X_c\) only as a class for the initial data. Note that \(\mathcal X_c\) is \textit{not} invariant under the affine semigroup \(S_t\). In general, \(S_t\rho\) has boundary value \(c\) but is not necessarily identically \(c\) in a fixed neighbourhood of \(\partial\Omega\).

We fix \(c\ge0\), and let \(P_t\) denote the homogeneous Dirichlet heat semigroup
(the case \(c=0\)). Whenever \(\rho_0-c\in L^1(\Omega)\), we define the associated
affine semigroup by
\begin{equation}\label{eq:def-St}
S_t\rho_0:=c+P_t(\rho_0-c).
\end{equation}
For \(\rho_0\in\mathcal D_c\), this is the classical solution of the Dirichlet
problem with a constant boundary value \(c\), that is,
\begin{equation}\label{eq:dirichlet}
\begin{aligned}
\partial_t\rho&=\Delta\rho &&\text{in }\Omega\times(0,\infty),\\
\rho(t,\cdot)&=c &&\text{on }\partial\Omega\times(0,\infty),\\
\rho(0,\cdot)&=\rho_0.
\end{aligned}
\end{equation}

\begin{definition}[Action of \texorpdfstring{$P_t$}{Pt} on measures]\label{def:Pt-on-measures}
For \(\mu\in\mathcal M(\Omega)\) and \(t\ge0\), we define \(P_t\mu\in\mathcal M(\Omega)\) by
\[
\int_\Omega \varphi\,d(P_t\mu)=\int_\Omega (P_t\varphi)\,d\mu
\qquad\forall \varphi\in C_{\partial\Omega}(\overline\Omega),
\]
where
\[
C_{\partial\Omega}(\overline\Omega)
:=\{f\in C(\overline\Omega): f|_{\partial\Omega}=0\}.
\]
We identify this space isometrically with \(C_0(\Omega)\), the space of
continuous functions on \(\Omega\) that vanish at the boundary, by restriction to \(\Omega\). Given that the killed heat semigroup is a positive contraction on this
\(C_0(\Omega)\), the right-hand side defines a positive bounded linear
functional on \(C_{\partial\Omega}(\overline\Omega)\), that is, a unique finite nonnegative Radon measure by the Riesz representation theorem. If \(\mu=\rho\,dx\) with \(\rho\in L^1(\Omega)\), our next lemma shows, for the Dirichlet Laplacian, that this dual definition agrees with the usual \(L^1\)-Dirichlet semigroup.
We use here the standard sub-Markovian Dirichlet heat semigroup on
\(C_0(\Omega)\), namely: positivity, \(C_0\)-contractivity, and strong continuity (see, for example, \cite[Chapters~1--2]{Ouh05}). In particular, the dual action maps finite
positive Radon measures to finite positive Radon measures and does not increase mass. For absolutely continuous data, we also write
\[
S_t\mu:=(S_t\rho)\,dx
\qquad\text{when }\mu=\rho\,dx.
\]
\end{definition}

\begin{lemma}[Duality for the Dirichlet Laplacian on \(L^1\)-\(C_0\)]\label{lem:duality-L1-C0}
For every \(t\ge0\), every \(f\in L^1(\Omega)\), and every
\(\varphi\in C_{\partial\Omega}(\overline\Omega)\),
\[
\int_\Omega \varphi\,P_t f\,dx
=
\int_\Omega (P_t\varphi)\,f\,dx.
\]
\end{lemma}

\begin{proof}
Let us choose \(f_k\in C_c^\infty(\Omega)\) with \(f_k\to f\) in \(L^1(\Omega)\), and choose \(\varphi_k\in C_c^\infty(\Omega)\) with \(\varphi_k\to\varphi\) uniformly on \(\overline\Omega\). The latter is possible because \(C_c^\infty(\Omega)\) is uniformly dense in \(C_{\partial\Omega}(\overline\Omega)\) with respect to the uniform norm. By the self-adjointness of the Dirichlet Laplacian on \(L^2(\Omega)\), we have
\[
\int_\Omega \varphi_k\,P_t f_k\,dx
=
\int_\Omega (P_t\varphi_k)\,f_k\,dx.
\]
Because \(P_t\) is a contraction on \(L^1(\Omega)\) and on \(C_{\partial\Omega}(\overline\Omega)\), both sides converge to the limits in question.
\end{proof}

\subsection{A Hopf-type boundary lower bound}

\begin{lemma}[Uniform linear lower bound near the boundary]\label{lem:hopf-linear-general}
Let
\[
\mathcal Au := \nabla\cdot(A(x)\nabla u)+b(x)\cdot\nabla u+q(x)\,u
\]
on $\Omega$, where \(A\in C^\infty(\overline\Omega;\mathbb R^{n\times n})\),
\(b\in C^\infty(\overline\Omega;\mathbb R^n)\), and \(q\in C^\infty(\overline\Omega)\).
We do not assume that the matrix \(A\) is symmetric. We write
\[
A_s:=\frac{A+A^{\mathsf T}}2,
\]
and assume that there exist \(0<\lambda\le \Lambda<\infty\) such that
\[
\lambda|\xi|^2\le \xi\cdot A_s(x)\xi\le \Lambda|\xi|^2,
\qquad
\|A(x)\|\le \Lambda,
\]
for every \(x\in\overline\Omega\) and every \(\xi\in\mathbb R^n\). Let $\eta\in C^\infty(\overline\Omega)$ satisfy $\eta\ge0$, $\eta\not\equiv0$, and $\eta=0$ on $\partial\Omega$. We further take $t_0>0$, and let $w(t,\cdot):=P_t^{\mathcal A}\eta$, where $P_t^{\mathcal A}$ is the homogeneous Dirichlet semigroup generated by $\mathcal A$. For each connected component \(\Omega_j\subset\Omega\), the restriction \(w(t_0,\cdot)\restriction_{\Omega_j}\) extends to a \(C^1\)-function on \(\overline{\Omega_j}\). Moreover, for every connected component $\Omega'\subset\Omega$ such that $\eta\not\equiv0$ on $\Omega'$, we have $w(t_0,x)>0$ for all $x\in\Omega'$, and there exist constants $\alpha=\alpha(t_0,\eta,\mathcal A,\Omega')>0$ and $r=r(t_0,\eta,\mathcal A,\Omega')>0$ such that
\[
w(t_0,x)\ge \alpha\,\delta(x)\qquad\text{for all }x\in\Omega'\text{ with }\delta(x)\le r.
\]
On an arbitrary component $\Omega''$ with $\eta\equiv0$, it holds that $w(t,\cdot)\equiv0$ on $\Omega''$ for all $t\ge0$.
\end{lemma}

\begin{proof}
The statement follows componentwise. If \(\eta\equiv0\) on a component, by uniqueness for the homogeneous Dirichlet problem, \(w\equiv0\) holds there. We fix a connected component \(\Omega'\) on which \(\eta\not\equiv0\). We defer the proof details to Appendix~\ref{app:hopf-detail}. It gives \(w(t_0,\cdot)\in C^1(\overline{\Omega'})\), strict positivity of \(w\) in \((0,t_0]\times\Omega'\), and
\(\partial_{\mathbf n}w(t_0,\xi)>0\) for every \(\xi\in\partial\Omega'\), where
\(\mathbf n\) is the inward unit normal. We now apply compactness of \(\partial\Omega'\) and the tubular representation of Lemma~\ref{lem:tubular} to \(\Omega'\), and obtain constants \(\alpha,r>0\) such that
\[
w(t_0,x)\ge \alpha\,\delta(x)
\qquad\text{for all }x\in\Omega'\text{ with }\delta(x)\le r .
\]
In Appendix~\ref{app:hopf-detail}, we also prove the \(C^1\)-extension on components on which \(\eta\equiv0\).
\end{proof}

\begin{proposition}[Boundary-layer test measures for the Dirichlet Laplacian]\label{prop:amplification-packet-ell}
We fix \(p>1\). For every \(t>0\), there exist constants \(\alpha>0\) and \(r_*>0\) with the following property: for every \(\varepsilon\in(0,r_*/4)\), there exists \(u_0^\varepsilon\in C_c^\infty(\Omega)\), \(u_0^\varepsilon\ge0\), such that
\[
\int_\Omega u_0^\varepsilon\,dx=1,
\qquad
\supp(u_0^\varepsilon)\subset\{x\in\Omega:\ \varepsilon\le \delta(x)\le 2\varepsilon\}.
\]
If we set \(\mu_\varepsilon:=u_0^\varepsilon\,dx\), we obtain
\[
W_{b,p}(\mu_\varepsilon,0)^p\le (2\varepsilon)^p,
\qquad
W_{b,p}\bigl(P_t\mu_\varepsilon,0\bigr)^p
\ge \alpha\,\varepsilon.
\]
\end{proposition}

\begin{proof}
It suffices to construct the witnesses inside one connected component, because
extension by zero to the other components preserves admissibility as
measures on \(\Omega\). It also preserves the boundary distance by Lemma~\ref{lem:components}. Let us fix such a component \(\Omega'\subset\Omega\), and choose
\(\varphi\in C_c^\infty(\Omega')\) with \(\varphi\ge0\) and \(\varphi\not\equiv0\).
Since \(\supp(\varphi)\Subset\Omega'\), we get
\[
\delta_*:=\min_{x\in\supp(\varphi)}\delta(x)>0.
\]
We also set
\[
\kappa:=\frac{\delta_*^p}{2\|\varphi\|_{L^\infty(\Omega)}},
\qquad
\eta:=\kappa\varphi.
\]
We extend \(\eta\) by zero outside \(\Omega'\), and use
\(\supp\varphi\Subset\Omega'\), which gives
\[
        \eta\in C_c^\infty(\Omega)\subset C^\infty(\overline\Omega),
        \qquad
        \eta\ge0,
        \qquad
        \eta\not\equiv0,
        \qquad
        \eta=0 \ \text{on } \partial\Omega .
\]
Moreover,
\[
0\le \eta(x)\le \delta(x)^p\qquad\forall x\in\Omega.
\]
Let \(w:=P_t\eta\). We apply Lemma~\ref{lem:hopf-linear-general} with
\(\mathcal A=\Delta\) on the component \(\Omega'\), which gives constants
\(\alpha>0\) and \(r>0\) such that
\[
w(x)\ge \alpha\,\delta(x)\qquad\text{for all }x\in\Omega'\text{ with }\delta(x)\le r.
\]
Let \(\varepsilon_0^{\mathrm{coll}}(\Omega')>0\) be the constant from Lemma~\ref{lem:collar-bump} that we apply to the bounded \(C^2\) domain \(\Omega'\), and set
\[
r_*:=\min\{r,4\varepsilon_0^{\mathrm{coll}}(\Omega')\}.
\]
We fix \(\xi_0\in\partial\Omega'\). For each \(\varepsilon\in(0,r_*/4)\), Lemma~\ref{lem:collar-bump}, which we apply on \(\Omega'\), yields a function
\(u_0^\varepsilon\in C_c^\infty(\Omega')\), \(u_0^\varepsilon\ge0\), such that
\[
\int_{\Omega'}u_0^\varepsilon\,dx=1
\]
and
\[
\begin{aligned}
\supp(u_0^\varepsilon)
&\subset
B_{\varepsilon/4}\!\left(\xi_0+\frac{3\varepsilon}{2}\,\mathbf n(\xi_0)\right)\\
&\subset
\{x\in\Omega':\ \varepsilon\le \dist(x,\partial\Omega')\le 2\varepsilon\}.
\end{aligned}
\]
By Lemma~\ref{lem:components}, \(\dist(x,\partial\Omega')=\delta(x)\) on \(\Omega'\). If we extend \(u_0^\varepsilon\) by \(0\) to \(\Omega\), we still have \(u_0^\varepsilon\in C_c^\infty(\Omega)\), \(\int_\Omega u_0^\varepsilon\,dx=1\), and
\[
\supp(u_0^\varepsilon)
\subset
\Omega'\cap\{x\in\Omega:\ \varepsilon\le \delta(x)\le 2\varepsilon\}.
\]
We also set \(\mu_\varepsilon:=u_0^\varepsilon\,dx\). By Lemma~\ref{lem:tozero},
\[
W_{b,p}(\mu_\varepsilon,0)^p
=\int_\Omega \delta(x)^p\,u_0^\varepsilon(x)\,dx
\le (2\varepsilon)^p.
\]
By Lemma~\ref{lem:duality-L1-C0}, \(P_t\mu_\varepsilon=(P_tu_0^\varepsilon)\,dx\).
By positivity, we obtain \(P_tu_0^\varepsilon\ge0\). If we now use Lemma~\ref{lem:tozero}, the bound \(0\le\eta\le\delta^p\), and Lemma~\ref{lem:duality-L1-C0}, we obtain
\[
\begin{aligned}
W_{b,p}\bigl(P_t\mu_\varepsilon,0\bigr)^p
&=\int_\Omega \delta(x)^p\,P_tu_0^\varepsilon(x)\,dx\\
&\ge \int_\Omega \eta(x)\,P_tu_0^\varepsilon(x)\,dx\\
&=\int_\Omega (P_t\eta)(x)\,u_0^\varepsilon(x)\,dx\\
&=\int_\Omega w(x)\,u_0^\varepsilon(x)\,dx.
\end{aligned}
\]
Since \(\supp(u_0^\varepsilon)\subset\Omega'\cap\{\varepsilon\le\delta\le2\varepsilon\}\subset\Omega'\cap\{\delta\le r\}\), the lower bound for \(w\), which we just established, gives
\[
W_{b,p}\bigl(P_t\mu_\varepsilon,0\bigr)^p
\ge \alpha\varepsilon\int_\Omega u_0^\varepsilon(x)\,dx
=\alpha\varepsilon.
\]
\end{proof}

\begin{proposition}[Laplacian lower-bound results for
\texorpdfstring{$p>1$}{p>1}]\label{prop:Delta-lower-consequences}
For the Dirichlet Laplacian, we fix \(p>1\) and \(t>0\). It follows that the map
\[
P_t:\mathcal M(\Omega)\to\mathcal M(\Omega)
\]
is discontinuous at \(0\) with respect to \(W_{b,p}\). Moreover, for every
\(m>0\),
\[
\Lip_{W_{b,p}}\!\bigl(P_t;\mathcal M_{\le m}(\Omega)\bigr)=\infty.
\]
\end{proposition}

\begin{proof}
We apply Proposition~\ref{prop:amplification-packet-ell}. For the witness measures
\(\mu_\varepsilon\),
\[
\frac{W_{b,p}\bigl(P_t\mu_\varepsilon,0\bigr)}
{W_{b,p}(\mu_\varepsilon,0)}
\ge
\frac{\alpha^{1/p}}{2}\,\varepsilon^{-(1-1/p)}
\xrightarrow[\varepsilon\downarrow0]{}+\infty.
\]
It suffices to scale by \(m>0\) and apply Remark~\ref{rem:Wbp-homogeneity} to obtain the infinite Lipschitz constant on every total-mass sublevel \(\mathcal M_{\le m}(\Omega)\). For discontinuity at \(0\), we set \(\tau_\varepsilon:=\varepsilon^{-1}\mu_\varepsilon\). We then have
\[
W_{b,p}(\tau_\varepsilon,0)^p\le 2^p\varepsilon^{p-1}\to0,
\qquad
W_{b,p}\bigl(P_t\tau_\varepsilon,0\bigr)^p\ge\alpha.
\]
\end{proof}

\section{\texorpdfstring{$W_{b,1}$}{Wb,1}-Lipschitz regularity at the endpoint \texorpdfstring{$p=1$}{p=1}}\label{sec:Wb1}

In the endpoint argument, we use the following \(p=1\) dual formula, which is the collapsed-boundary specialisation of the relative Kantorovich--Rubinstein duality for metric pairs (\textit{cf}. \cite[Theorem~7.6]{BE24}). We present the proof in Appendix~\ref{app:Wb1-shortcut} to fix the normalisation.

\begin{theorem}[Dual formula for \texorpdfstring{$W_{b,1}$}{Wb,1}]\label{thm:KR-Wb1}
For all $\mu,\nu\in\mathcal M(\Omega)$,
\begin{equation}\label{eq:KR-Wb1}
W_{b,1}(\mu,\nu)
=
\sup\left\{
\int_\Omega \varphi\,d(\mu-\nu):
\ \varphi\in C(\overline\Omega),\
\varphi|_{\partial\Omega}=0,\
\Lip(\varphi)\le 1
\right\}.
\end{equation}
\end{theorem}

\begin{proof}
See Appendix~\ref{app:Wb1-shortcut}.
\end{proof}

The admissible test class in \eqref{eq:KR-Wb1} is symmetric under
\(\varphi\mapsto-\varphi\). Note that the same formula remains the same, even if we put an absolute value around the integral.

\begin{corollary}[Scaling of the KR inequality]\label{cor:KR-scaling-Wb1}
For all $\mu,\nu\in\mathcal M(\Omega)$ and all $\psi\in C(\overline\Omega)$ with $\psi|_{\partial\Omega}=0$ and $\Lip(\psi)<\infty$,
\[
\left|\int_\Omega \psi\,d(\mu-\nu)\right|
\le \Lip(\psi)\,W_{b,1}(\mu,\nu).
\]
\end{corollary}

\begin{proof}
If $\Lip(\psi)=0$, then $\psi$ is constant on $\overline\Omega$, so $\psi\equiv0$ because $\psi|_{\partial\Omega}=0$. If $\Lip(\psi)>0$, we set $\phi:=\psi/\Lip(\psi)$.
It follows that $\phi$ and $-\phi$ are allowed test functions in \eqref{eq:KR-Wb1}, so
\[
\left|\int_\Omega \phi\,d(\mu-\nu)\right|\le W_{b,1}(\mu,\nu).
\]
It suffices to multiply by $\Lip(\psi)$ to obtain the claim.
\end{proof}

\begin{lemma}\label{lem:grad-implies-lip-zero-trace}
Let $u\in C(\overline\Omega)\cap C^1(\Omega)$ satisfy $u|_{\partial\Omega}=0$, and set $M:=\|\nabla u\|_{L^\infty(\Omega)}$. This implies that $u$ is Euclidean $M$-Lipschitz on $\overline\Omega$, that is,
\[
\Lip(u)\le M.
\]
\end{lemma}

\begin{proof}
If $M=+\infty$, the estimate follows immediately. We fix $x\in\Omega$ and choose $\xi\in\partial\Omega$ with $|x-\xi|=\delta(x)$. For $s\in[0,1)$, define $\gamma(s):=x+s(\xi-x)$. If $\gamma(s_0)\notin\Omega$ for some $s_0\in(0,1)$, by continuity of $\gamma$, we obtain $s_*\in(0,s_0]$ with
$\gamma(s_*)\in\partial\Omega$, so
$\delta(x)\le |x-\gamma(s_*)|=s_*|x-\xi|<|x-\xi|=\delta(x)$, which leads to a contradiction. Therefore, $\gamma([0,1))\subset\Omega$ holds, and by the fundamental theorem of calculus, we have
\[
u(\gamma(s))-u(x)=\int_0^s \nabla u(\gamma(r))\cdot(\xi-x)\,dr.
\]
We let $s\uparrow1$ and use $u(\xi)=0$, which gives
\[
|u(x)|\le M\,\delta(x).
\]
This also holds for $x\in\partial\Omega$, where both sides are $0$. We now fix $x,y\in\overline\Omega$, $x\neq y$. If $x,y\in\partial\Omega$, then $u(x)=u(y)=0$.
Let $\Gamma:=\{x+s(y-x): s\in[0,1]\}$. If $\Gamma\subset\Omega$, then
\[
\begin{aligned}
|u(y)-u(x)|
&\le \int_0^1 |\nabla u(x+s(y-x))|\,|x-y|\,ds\\
&\le M|x-y|.
\end{aligned}
\]
If $\Gamma\not\subset\Omega$, we pick an arbitrary $z\in\Gamma\cap\partial\Omega$. It follows that $u(z)=0$, and the pointwise bound above gives
\[
|u(x)-u(y)|\le |u(x)|+|u(y)|\le M(\delta(x)+\delta(y)).
\]
Observe that $z$ lies on the segment, so $|x-y|=|x-z|+|z-y|$ holds, while
$\delta(x)\le |x-z|$ and $\delta(y)\le |y-z|$. Hence $\delta(x)+\delta(y)\le |x-y|$, so, again, $|u(x)-u(y)|\le M|x-y|$. We now take the supremum over $x\neq y$ to obtain $\Lip(u)\le M$.
\end{proof}

\begin{lemma}[Positive-time global parabolic regularity]
\label{lem:positive-time-regularity}
Let \(D\subset\mathbb R^n\) be a bounded \(C^2\) domain and let
\[
\partial_t u-\sum_{i,j}a_{ij}(x)\partial_{ij}u
-\sum_i b_i(x)\partial_i u-c(x)u=F
\]
be uniformly parabolic on \((a,b)\times D\), with smooth bounded
coefficients, zero Dirichlet boundary condition, and a zero initial datum at
\(t=a\). Moreover, let \(u\) be the unique distributional \(L^r\)-solution of this
zero-initial Cauchy--Dirichlet problem. If \(F\in L^r((a,b)\times D)\),
\(1<r<\infty\), then \(u\in W^{2,1}_r((a,b)\times D)\) and satisfies the
standard estimate
\[
\|u\|_{W^{2,1}_r((a,b)\times D)}
\le C\|F\|_{L^r((a,b)\times D)}.
\]
For \(r>n+2\), by the parabolic Sobolev--Morrey embedding, we obtain, for every
\(\theta\in(0,(b-a)/2)\),
\[
\nabla_xu\in C^{\alpha,\alpha/2}
\bigl([a+\theta,b-\theta]\times\overline D\bigr)
\]
for some \(\alpha\in(0,1)\), with the norm controlled by the same data. Notice that the constant depends only on \(D,r,a,b\), the ellipticity constants, the \(C^2\)-atlas of \(\partial D\), and the coefficient norms and moduli of continuity that appear in the global \(W^{2,1}_r\) estimate. In the smooth-coefficient setting below, this means dependence on finitely many \(C^k\)-bounds of the coefficients on \(\overline D\), for \(k\) sufficiently large in the global estimate.
\end{lemma}

\begin{proof}
This is the standard global \(L^r\)-solvability and a priori estimate for
zero-initial, zero-lateral Cauchy--Dirichlet problems on \(C^{1,1}\) cylinders
(for instance, see Lieberman's global \(L^p\) parabolic Cauchy--Dirichlet estimates in
Chapter VII, Sections 7.1--7.4 of \cite{Lie96}). Because a compact \(C^2\)
boundary is locally \(C^{1,1}\), all the hypotheses of these estimates are satisfied after boundary flattening. The time cut-offs we apply below vanish near the initial time and on the lateral boundary, which means that the necessary compatibility conditions hold automatically. The Morrey conclusion follows directly from the parabolic Sobolev--Morrey embedding for \(W^{2,1}_r\) with \(r>n+2\).
\end{proof}

We now analyse the fixed-time gradient estimate that will be needed for the endpoint \(p=1\). We include a concise \(C^2\)-domain proof needed here, which is based on standard global parabolic \(W^{2,1}_r\)-regularity. \cite{Wang04} describes related uniform estimates for Dirichlet heat semigroups obtained by probabilistic methods and under geometric hypotheses.

\begin{theorem}[Fixed-time Dirichlet gradient estimate]\label{thm:dirichlet-grad}
Let \(\Omega\subset\mathbb R^n\) be a nonempty bounded open set with \(C^2\)
boundary, and let \((P_t)_{t\ge0}\) be the homogeneous Dirichlet heat semigroup.
For every \(t_0>0\), there exists \(C_{t_0,\Omega}<\infty\) such that
\[
\|\nabla P_{t_0} f\|_{L^\infty(\Omega)}
\le C_{t_0,\Omega}\|f\|_{L^\infty(\Omega)}
\qquad\forall f\in L^\infty(\Omega).
\]
\end{theorem}

\begin{proof}
We refer to the standard global \(W^{2,1}_r\) Cauchy--Dirichlet estimate on
bounded \(C^2\) cylinders, as in Lieberman's parabolic boundary regularity
theory \cite{Lie96}. By Lemma~\ref{lem:components}, it is enough to prove the estimate on each connected component and then take the maximum of the new constants. To this end, we assume first that \(\Omega\) is connected. We fix \(t_0>0\) and set, for definiteness, \(\tau:=t_0/4\). We also choose
\(\chi\in C_c^\infty((t_0-\tau,t_0+\tau))\) such that \(0\le\chi\le1\) and
\[
\chi\equiv1
\qquad\text{on }[t_0-\tau/2,t_0+\tau/2].
\]
In this step, we extend \(f\) by \(0\) outside \(\Omega\), and choose cutoffs
\(\chi_k\in C_c^\infty(\Omega)\) with \(0\le\chi_k\le1\) and
\(\chi_k\to1\) a.e. in \(\Omega\). After mollification of \(\chi_k f\) and truncation,
we obtain \(f_k\in C_c^\infty(\Omega)\) such that
\[
\|f_k\|_{L^\infty}\le \|f\|_{L^\infty}+k^{-1},
\qquad
f_k\to f \quad\text{in }L^r(\Omega)
\]
for every finite \(r\). It is enough to establish the estimate first for smooth
data, where the constants are independent of the datum, and then pass to the limit. Let us assume for now that \(f\in C_c^\infty(\Omega)\). Let
\(w(t,x):=P_t f(x)\) and set
\[
u(t,x):=\chi(t)w(t,x).
\]
It follows that \(u\) is a distributional solution on
\[
Q:=(t_0-\tau,t_0+\tau)\times\Omega
\]
of the homogeneous Dirichlet problem
\[
\partial_t u-\Delta u=(\partial_t\chi)w,
\qquad
u|_{\partial\Omega}=0,
\qquad
u(t_0-\tau,\cdot)=0.
\]
Since \(P_t\) is an \(L^\infty\)-contraction,
\[
\|w\|_{L^\infty(Q)}\le \|f\|_{L^\infty(\Omega)}.
\]
Hence, for every \(1<r<\infty\),
\[
\|(\partial_t\chi)w\|_{L^r(Q)}
\le C(t_0,\tau,r,\Omega)\|f\|_{L^\infty(\Omega)}.
\]
Because the time cut-off \(\chi\) is compactly supported in
\((t_0-\tau,t_0+\tau)\), the zero initial condition at \(t=t_0-\tau\) and the
boundary compatibility conditions hold automatically. By Lemma~\ref{lem:positive-time-regularity}, which we apply to the time-cutoff solution with \(F=(\partial_t\chi)w\), and by the standard energy/duality uniqueness for the zero-initial Cauchy--Dirichlet problem in \(L^r\), we obtain
\[
\|u\|_{W^{2,1}_r(Q)}
\le C(t_0,\tau,r,\Omega)\|f\|_{L^\infty(\Omega)}.
\]
We choose \(r>n+2\). By the parabolic Morrey embedding for the
\(W^{2,1}_r\)-scale we have, on the smaller slab,
\[
\nabla_x u\in C^{\alpha,\alpha/2}
\bigl([t_0-\tau/2,t_0+\tau/2]\times\overline\Omega\bigr)
\]
for some \(\alpha\in(0,1)\), and, in particular,
\[
\|\nabla_x u\|_{L^\infty(
[t_0-\tau/2,t_0+\tau/2]\times\Omega)}
\le C\,\|u\|_{W^{2,1}_r(Q)}
\le C(t_0,\Omega)\|f\|_{L^\infty(\Omega)}.
\]
Because \(u=w\) on the smaller slab, this gives
\[
\|\nabla P_{t_0}f\|_{L^\infty(\Omega)}
\le C(t_0,\Omega)\|f\|_{L^\infty(\Omega)}.
\]
For general \(f\in L^\infty(\Omega)\), we apply the previous estimate to the
sequence \(f_k\). Given that the Dirichlet heat semigroup is sub-Markovian, it is
contractive on \(L^r(\Omega)\) for every \(1\le r\le\infty\). It follows that
\[
P_{t_0}f_k\to P_{t_0}f
\qquad\text{in }L^r(\Omega)
\]
for every finite \(r\). Note that the gradients \(\nabla P_{t_0}f_k\) are bounded in
\(L^\infty(\Omega;\mathbb R^n)\), so, after we pass to a subsequence,
converge weakly-* in \(L^\infty(\Omega;\mathbb R^n)\) to some vector field
\(G\). Since \(P_{t_0}f_k\to P_{t_0}f\) in \(L^r(\Omega)\), the distributional identity \(G=\nabla P_{t_0}f\) follows by testing against \(C_c^\infty(\Omega)\). We conclude that \(P_{t_0}f\in W^{1,\infty}(\Omega)\), and by weak-* lower semicontinuity, we have
\[
\|\nabla P_{t_0}f\|_{L^\infty(\Omega)}
\le C(t_0,\Omega)\|f\|_{L^\infty(\Omega)}.
\]
Initially, this estimate is an essential \(L^\infty\)-bound. Since the Dirichlet heat semigroup has an interior \(C^\infty\)-representative for
\(t_0>0\), the distributional gradient coincides in \(\Omega\) with the classical
gradient of this representative. Moreover, the classical gradient is continuous on compact subsets of \(\Omega\), so the essential bound extends pointwise in \(\Omega\). Otherwise, continuity at a point of excess would contradict the essential
\(L^\infty\)-bound.

For a general bounded \(C^2\) open set, we apply the previous argument on each
connected component \(\Omega_j\). Given that there are only finitely many components by Lemma~\ref{lem:components}, it suffices to take the maximum of the componentwise constants to obtain the desired estimate.
\end{proof}

\begin{proposition}[Propagation of boundary Lipschitz test functions]\label{prop:Pt-test-lip}
For every \(t>0\), there exists \(L_{t,\Omega}<\infty\) such that, whenever \(\varphi\in C(\overline\Omega)\) satisfies \(\varphi|_{\partial\Omega}=0\) and \(\Lip(\varphi)\le1\), it holds that
\[
P_t\varphi\in C_{\partial\Omega}(\overline\Omega)\cap C^1(\Omega),
\qquad
\Lip(P_t\varphi)\le L_{t,\Omega}.
\]
\end{proposition}

\begin{proof}
Let \(C_{t,\Omega}\) be as in Theorem~\ref{thm:dirichlet-grad}, and set
\[
R_\Omega:=\sup_{x\in\Omega}\delta(x)<\infty,
\qquad
L_{t,\Omega}:=C_{t,\Omega}R_\Omega.
\]
If \(\varphi|_{\partial\Omega}=0\) and \(\Lip(\varphi)\le1\), then
\[
|\varphi(x)|\le \delta(x)\le R_\Omega
\qquad\forall x\in\Omega,
\]
so \(\|\varphi\|_{L^\infty(\Omega)}\le R_\Omega\). Since \(P_t\) is the killed heat semigroup on \(C_0(\Omega)\), we have \(P_t\varphi\in C_{\partial\Omega}(\overline\Omega)\). Moreover, for every \(t>0\), by the standard positive-time interior smoothing for the Dirichlet heat semigroup, we obtain \(P_t\varphi\in C^\infty(\Omega)\). The boundary-continuous representative lies in
\(C_{\partial\Omega}(\overline\Omega)\) by the \(C_0(\Omega)\)-semigroup
realisation. Finally, by Theorem~\ref{thm:dirichlet-grad}, we have
\[
\|\nabla P_t\varphi\|_{L^\infty(\Omega)}
\le C_{t,\Omega}\|\varphi\|_{L^\infty(\Omega)}
\le L_{t,\Omega}.
\]
We conclude that \(P_t\varphi\in C^1(\Omega)\), and Lemma~\ref{lem:grad-implies-lip-zero-trace} gives
\[
\Lip(P_t\varphi)\le L_{t,\Omega}.
\]
\end{proof}

\begin{theorem}[Endpoint \texorpdfstring{$p=1$}{p=1}]\label{thm:p1-finite-lip}
For every $t>0$, there exists $L_{t,\Omega}\in(0,\infty)$ such that
\[
W_{b,1}(P_t\mu,P_t\nu)\le L_{t,\Omega}\,W_{b,1}(\mu,\nu)
\qquad \forall \mu,\nu\in\mathcal M(\Omega).
\]
In particular,
\[
\Lip_{W_{b,1}}(P_t)<\infty.
\]
\end{theorem}

\begin{proof}[Proof of Theorem~\ref{thm:p1-finite-lip}]
Let \(L_{t,\Omega}\) be as in Proposition~\ref{prop:Pt-test-lip}. By Theorem~\ref{thm:KR-Wb1},
\[
W_{b,1}(P_t\mu,P_t\nu)
=
\sup_{\substack{\varphi\in C(\overline\Omega)\\ \varphi|_{\partial\Omega}=0,\ \Lip(\varphi)\le1}}
\int_\Omega \varphi\,d(P_t\mu-P_t\nu).
\]
For every such \(\varphi\), Definition~\ref{def:Pt-on-measures} gives
\[
\int_\Omega \varphi\,d(P_t\mu-P_t\nu)=\int_\Omega (P_t\varphi)\,d(\mu-\nu).
\]
By Proposition~\ref{prop:Pt-test-lip}, the function \(P_t\varphi\) vanishes on \(\partial\Omega\) and satisfies \(\Lip(P_t\varphi)\le L_{t,\Omega}\). By Corollary~\ref{cor:KR-scaling-Wb1}, we have
\[
\left|\int_\Omega (P_t\varphi)\,d(\mu-\nu)\right|
\le L_{t,\Omega}\,W_{b,1}(\mu,\nu).
\]
Finally, we take the supremum over \(\varphi\) to obtain the estimate.
\end{proof}

\begin{proof}[Proof of Theorem~\ref{thm:sharp-modulus}]
The \(p=1\) statement is precisely Theorem~\ref{thm:p1-finite-lip}. Let \(p>1\), \(m>0\), and \(\mu,\nu\in\mathcal M_{\le m}(\Omega)\). By Lemma~\ref{lem:Wb1-Wbp-comparison}, Theorem~\ref{thm:p1-finite-lip}, and, again, Lemma~\ref{lem:Wb1-Wbp-comparison},
\[
\begin{aligned}
W_{b,p}(P_t\mu,P_t\nu)^p
&\le D_\Omega^{p-1} W_{b,1}(P_t\mu,P_t\nu)\\
&\le D_\Omega^{p-1} L_{t,\Omega} W_{b,1}(\mu,\nu)\\
&\le D_\Omega^{p-1} L_{t,\Omega}(2m)^{1-1/p}
      W_{b,p}(\mu,\nu).
\end{aligned}
\]
It follows that the \(1/p\)-H\"older estimate holds with
\[
C_{t,p,m,\Omega}
=
D_\Omega^{p-1}L_{t,\Omega}(2m)^{1-1/p}.
\]
Sharpness follows directly from the boundary-layer packets of Proposition~\ref{prop:amplification-packet-ell} for the Dirichlet Laplacian.
Let \(\mu_\varepsilon\) be the packets with unit mass and write \(\alpha>0\) for the lower-bound constant in that proposition. We also fix \(a\in(0,m]\). It follows that
\[
W_{b,p}(a\mu_\varepsilon,0)
\le 2a^{1/p}\varepsilon,
\]
whereas
\[
W_{b,p}(P_t(a\mu_\varepsilon),0)
=
a^{1/p}W_{b,p}(P_t\mu_\varepsilon,0)
\ge a^{1/p}\alpha^{1/p}\varepsilon^{1/p}.
\]
This implies that for every \(\alpha_0>1/p\),
\[
\frac{W_{b,p}(P_t(a\mu_\varepsilon),0)}
     {W_{b,p}(a\mu_\varepsilon,0)^{\alpha_0}}
\ge
\frac{a^{(1-\alpha_0)/p}\alpha^{1/p}}{2^{\alpha_0}}
\varepsilon^{1/p-\alpha_0}
\to+\infty
\qquad(\varepsilon\downarrow0).
\]
This proves optimality of the exponent \(1/p\) within the scale of power
moduli. Note that the Lipschitz failure is the special case \(\alpha_0=1\). Finally, we obtain the discontinuity at \(0\) on the full finite-measure space by setting \(\tau_\varepsilon:=\varepsilon^{-1}\mu_\varepsilon\), exactly as in
Proposition~\ref{prop:Delta-lower-consequences}.
\end{proof}

\begin{remark}\label{rem:p1-convex-ps}
In the convex Euclidean context, it is possible to sharpen the endpoint result to a
contraction. Under the hypotheses of \cite[Theorem~1.19]{PS20},
\cite[Corollary~1.24]{PS20} gives the contraction for the Profeta--Sturm
Dirichlet metrics. At \(p=1\), their results give
\[
W_1^\flat=W_1^\sharp=W'_1,
\qquad
W'_1=W''_1,
\]
and \(W''_1\) is the boundary-reservoir metric \(W_{b,1}\) in the
normalisation we use here. In the convex settings considered by \cite{PS20},
\[
W_{b,1}(P_t\mu,P_t\nu)\le W_{b,1}(\mu,\nu)
\]
for the homogeneous Dirichlet heat flow on subprobabilities. By the positive
homogeneity of \(W_{b,1}\) and the linearity of \(P_t\), the same estimate
extends to arbitrary finite nonnegative Radon measures after rescaling. The relevance of Theorem~\ref{thm:p1-finite-lip} is that it could be applied on arbitrary
bounded \(C^2\) domains, without any convexity or curvature assumptions.
\end{remark}

We now analyse the affine endpoint \(p=1\).

\begin{corollary}[Global Lipschitz bound for the affine Dirichlet flow at \texorpdfstring{$p=1$}{p=1}]\label{cor:p1-affine}
We fix $c\ge0$ and $t>0$. There exists $L_{t,\Omega}\in(0,\infty)$ such that
\[
W_{b,1}\bigl((S_t\rho)\,dx,(S_t\sigma)\,dx\bigr)
\le L_{t,\Omega}\,W_{b,1}(\rho\,dx,\sigma\,dx)
\qquad \forall \rho,\sigma\in\mathcal D_c.
\]
In particular,
\[
\Lip_{W_{b,1}}\!\left(S_t;\mathcal X_c\right)<\infty.
\]
\end{corollary}

\begin{proof}
Let \(L_{t,\Omega}\) be as in Proposition~\ref{prop:Pt-test-lip}.
For \(\rho,\sigma\in\mathcal D_c\) and \(\varphi\in C(\overline\Omega)\) with \(\varphi|_{\partial\Omega}=0\), Lemma~\ref{lem:duality-L1-C0} gives
\[
\int_\Omega \varphi\,d\bigl((S_t\rho)\,dx-(S_t\sigma)\,dx\bigr)
=
\int_\Omega (P_t\varphi)(\rho-\sigma)\,dx.
\]
We may now apply the proof of Theorem~\ref{thm:p1-finite-lip} verbatim, where we replace Definition~\ref{def:Pt-on-measures} by this identity. This gives
\[
W_{b,1}\bigl((S_t\rho)\,dx,(S_t\sigma)\,dx\bigr)
\le L_{t,\Omega}\,W_{b,1}(\rho\,dx,\sigma\,dx).
\]
\end{proof}

\section{Uniformly elliptic robustness}\label{sec:ell-robustness}

\subsection{Uniformly elliptic Dirichlet semigroups and duality}
In the proof of Theorem~\ref{thm:sharp-modulus}, we use the specialisation
\(\mathcal L=\Delta\). The statements below concern lower-bound and instability results. We prove the exact matching \(1/p\)-H\"older upper modulus for the killed Dirichlet heat semigroup of the Laplacian, where the \(W_{b,1}\) endpoint estimate is available. In the extension from the Laplacian to a general uniformly elliptic operator, we use two results, namely, a Hopf-type lower bound for the adjoint flow and a duality identity.

Let \(\mathcal L\) be a time-independent second-order uniformly elliptic
operator on \(\Omega\) of the form
\begin{equation}\label{eq:ell-operator}
\mathcal L u := \nabla\cdot(A(x)\nabla u) + b(x)\cdot\nabla u,
\end{equation}
where \(A\in C^\infty(\overline\Omega;\mathbb R^{n\times n})\) and
\(b\in C^\infty(\overline\Omega;\mathbb R^n)\). We do not assume that the matrix \(A\) is symmetric. The forward operator does not contain a zero-order term. Indeed, zero-order terms appear only when the formal adjoint is rewritten in divergence form. We use this convention below in the affine constant-boundary subsection, where constants must be stationary. We write
\[
A_s:=\frac{A+A^{\mathsf T}}2,
\]
and assume that there exist constants \(0<\lambda\le \Lambda<\infty\) such that
\[
\lambda|\xi|^2\le \xi\cdot A_s(x)\xi\le \Lambda|\xi|^2,
\qquad
\|A(x)\|\le \Lambda,
\forall \xi\in\mathbb R^n,\ \forall x\in\overline\Omega.
\]
Let $(P_t^{\mathcal L})_{t\ge0}$ denote the homogeneous Dirichlet semigroup. For $\eta\in C^\infty(\overline\Omega)$ with $\eta|_{\partial\Omega}=0$,
$w(t,\cdot)=P_t^{\mathcal L}\eta$ is the classical solution of
\[
\begin{aligned}
\partial_t w&=\mathcal L w &&\text{in }\Omega\times(0,\infty),\\
w&=0 &&\text{on }\partial\Omega\times(0,\infty),\\
w(0,\cdot)&=\eta.
\end{aligned}
\]
We use the standard Dirichlet realisation of \(\mathcal L\), first on
\(L^2(\Omega)\) and then on \(L^1(\Omega)\) through the standard positive
extrapolated semigroup, with at most exponential growth (see, for
instance, \cite[Chapters~1--2 and~4]{Ouh05}). We invoke selected parabolic regularity statements from \cite[Chapter~7]{Lie96}. We denote the positive
strongly continuous semigroup by \((P_t^{\mathcal L})_{t\ge0}\).

In the uniformly elliptic part, we use this realisation through positivity on the nonnegative \(L^1\)-cone, \(L^1\)-boundedness on compact time intervals, the \(L^2\)-adjoint identity of Lemma~\ref{lem:duality-ell} for compactly supported smooth test functions, and the positive-time parabolic regularity given in
Lemma~\ref{lem:positive-time-regularity}. Whenever \(\mu=\rho\,dx\) with
\(\rho\in L^1(\Omega)\), we set
\[
P_t^{\mathcal L}\mu:=(P_t^{\mathcal L}\rho)\,dx.
\]
All statements which involve \(P_t^{\mathcal L}\) below are statements about
the density-induced map on the nonnegative \(L^1\)-cone. The formal adjoint of $\mathcal L$ (with respect to Lebesgue measure) is
\[
\begin{aligned}
\mathcal L^* v
&:=\nabla\cdot(A(x)^{\mathsf T}\nabla v)-\nabla\cdot(b(x)\,v)\\
&=\nabla\cdot(A^{\mathsf T}\nabla v)-b\cdot\nabla v-(\nabla\cdot b)\,v.
\end{aligned}
\]
We understand the Dirichlet realisation of \(\mathcal L^*\) as the \(L^2\)-adjoint of the Dirichlet realisation of \(\mathcal L\). Note that both realisations force the zero-trace boundary condition. We denote by $(P_t^{\mathcal L^*})_{t\ge0}$ the corresponding homogeneous Dirichlet semigroup.

\begin{lemma}[Duality]\label{lem:duality-ell}
For all $t\ge0$ and all $f,g\in C_c^\infty(\Omega)$, it holds that
\[
\int_\Omega (P_t^{\mathcal L} f)(x)\,g(x)\,dx
=\int_\Omega f(x)\,(P_t^{\mathcal L^*} g)(x)\,dx.
\]
\end{lemma}

\begin{proof}
The case \(t=0\) follows immediately. Let us fix \(t>0\) and define
\[
F(s):=\int_\Omega (P_s^{\mathcal L}f)(x)\,(P_{t-s}^{\mathcal L^*}g)(x)\,dx,
\qquad s\in[0,t].
\]
For \(s\in(0,t)\), both
\[
u_s:=P_s^{\mathcal L}f,
\qquad
v_s:=P_{t-s}^{\mathcal L^*}g
\]
lie in \(D(L_2)\) and \(D(L_2^*)\), respectively, by analyticity and
regularisation of the uniformly elliptic Dirichlet realisations. Here, \(L_2\)
is the \(L^2\)-Dirichlet realisation of \(\mathcal L\), and \(L_2^*\) is its
Hilbert-space adjoint. Furthermore, the maps \(s\mapsto u_s\) and \(s\mapsto v_s\)
are \(C^1\) as \(L^2\)-valued maps on compact subintervals of \((0,t)\), with
\[
\frac{d}{ds}u_s=L_2u_s,
\qquad
\frac{d}{ds}v_s=-L_2^*v_s.
\]
It follows that \(F\) is differentiable on \((0,t)\), and
\[
\begin{aligned}
F'(s)
&=\langle L_2u_s,v_s\rangle_{L^2}
  -\langle u_s,L_2^*v_s\rangle_{L^2}\\
&=0,
\end{aligned}
\]
by the adjointness of \(L_2\) and \(L_2^*\). We deduce that \(F\) is constant on
\((0,t)\).

Given that the Dirichlet realisations of \(\mathcal L\) and \(\mathcal L^*\) generate
strongly continuous semigroups on \(L^2(\Omega)\), we have
\[
P_s^{\mathcal L}f\to f,
\qquad
P_{t-s}^{\mathcal L^*}g\to P_t^{\mathcal L^*}g
\quad\text{in }L^2(\Omega)\ \text{as }s\downarrow0,
\]
and
\[
P_s^{\mathcal L}f\to P_t^{\mathcal L}f,
\qquad
P_{t-s}^{\mathcal L^*}g\to g
\quad\text{in }L^2(\Omega)\ \text{as }s\uparrow t.
\]
This implies that \(F\) extends continuously to \([0,t]\), with
\[
F(0)=\int_\Omega f(x)\,(P_t^{\mathcal L^*}g)(x)\,dx,
\qquad
F(t)=\int_\Omega (P_t^{\mathcal L}f)(x)\,g(x)\,dx.
\]
Since \(F\) is constant on \((0,t)\), these endpoint values coincide.
\end{proof}

\begin{remark}\label{rem:ell-positive}
If $f\ge0$ in $\Omega$, then $P_t^{\mathcal L}f\ge0$ for every $t\ge0$, by the
parabolic maximum principle for homogeneous Dirichlet problems. The same holds for $P_t^{\mathcal L^*}f$. If we write
\[
\begin{aligned}
\mathcal L^*v&=\nabla\cdot(A^{\mathsf T}\nabla v)-b\cdot\nabla v+qv,\\
q&:=-(\nabla\cdot b)\in C^\infty(\overline\Omega),
\end{aligned}
\]
and choose $K\ge \|q\|_{L^\infty(\Omega)}$, the function
$e^{-Kt}P_t^{\mathcal L^*}f$ solves a homogeneous Dirichlet problem with
zero-order coefficient $q-K\le0$. By the maximum principle, we obtain
$P_t^{\mathcal L^*}f\ge0$.
\end{remark}

\subsection{Infinite \texorpdfstring{$W_{b,p}$}{Wb,p}-Lipschitz constants for \texorpdfstring{$p>1$}{p>1}}

\begin{proposition}[Test measures on boundary layer]\label{prop:amplification-packet-uniform-ell}
We fix \(p>1\) and let \((P_t^{\mathcal L})_{t\ge0}\) be the homogeneous Dirichlet semigroup associated to \eqref{eq:ell-operator}. For every \(t>0\), there exist constants \(\alpha>0\) and \(r_*>0\) with the following property: for every \(\varepsilon\in(0,r_*/4)\), there exists \(u_0^\varepsilon\in C_c^\infty(\Omega)\), \(u_0^\varepsilon\ge0\), such that
\[
\int_\Omega u_0^\varepsilon\,dx=1,
\qquad
\supp(u_0^\varepsilon)\subset\{x\in\Omega:\ \varepsilon\le \delta(x)\le 2\varepsilon\}.
\]
If we set \(\mu_\varepsilon:=u_0^\varepsilon\,dx\), and write
\[
P_t^{\mathcal L}\mu_\varepsilon:=(P_t^{\mathcal L}u_0^\varepsilon)\,dx,
\]
we obtain
\[
W_{b,p}(\mu_\varepsilon,0)^p\le (2\varepsilon)^p,
\qquad
W_{b,p}\bigl((P_t^{\mathcal L}u_0^\varepsilon)\,dx,0\bigr)^p
\ge \alpha\,\varepsilon.
\]
\end{proposition}

\begin{proof}
The argument is the proof of Proposition~\ref{prop:amplification-packet-ell},
with the following substitutions. We choose the same component \(\Omega'\), the
same function \(\eta\in C_c^\infty(\Omega')\), \(\eta\ge0\), \(\eta\not\equiv0\),
with \(0\le\eta\le\delta^p\), and use the same collar construction for the
packets \(u_0^\varepsilon\). In the adjoint test step, we simply replace the self-adjoint identity for the Laplacian by Lemma~\ref{lem:duality-ell}, replace \(P_t\eta\) by \(P_t^{\mathcal L^*}\eta\), and apply Remark~\ref{rem:ell-positive} for positivity of
\(P_t^{\mathcal L}u_0^\varepsilon\), with Lemma~\ref{lem:hopf-linear-general} to \(\mathcal A=\mathcal L^*\) on \(\Omega'\). We set
\[
w:=P_t^{\mathcal L^*}\eta .
\]
By Lemma~\ref{lem:hopf-linear-general}, there exist \(\alpha>0\) and \(r>0\)
such that
\[
w(x)\ge \alpha\,\delta(x)
\qquad
\text{for all }x\in\Omega'\text{ with }\delta(x)\le r .
\]
Let \(\varepsilon_0^{\mathrm{coll}}(\Omega')>0\) be the constant from
Lemma~\ref{lem:collar-bump} that we apply to \(\Omega'\), and set
\[
r_*:=\min\{r,4\varepsilon_0^{\mathrm{coll}}(\Omega')\}.
\]
For each \(\varepsilon\in(0,r_*/4)\), by the collar construction, we obtain
\(u_0^\varepsilon\in C_c^\infty(\Omega')\), extended by zero to \(\Omega\), with
\[
\int_\Omega u_0^\varepsilon\,dx=1,
\qquad
\supp(u_0^\varepsilon)\subset
\{x\in\Omega':\varepsilon\le\delta(x)\le2\varepsilon\},
\]
where Lemma~\ref{lem:components} guarantees that \(\delta\) is computed in
\(\Omega'\) on this collar. By Lemma~\ref{lem:tozero}, we then get
\[
W_{b,p}(\mu_\varepsilon,0)^p
=\int_\Omega \delta^p u_0^\varepsilon\,dx
\le (2\varepsilon)^p .
\]
On the other hand, by positivity, Lemma~\ref{lem:tozero}, the bound
\(0\le\eta\le\delta^p\), and Lemma~\ref{lem:duality-ell}, it follows that
\[
\begin{aligned}
W_{b,p}\bigl((P_t^{\mathcal L}u_0^\varepsilon)\,dx,0\bigr)^p
&=\int_\Omega \delta(x)^p P_t^{\mathcal L}u_0^\varepsilon(x)\,dx\\
&\ge \int_\Omega \eta(x)P_t^{\mathcal L}u_0^\varepsilon(x)\,dx\\
&=\int_\Omega P_t^{\mathcal L^*}\eta(x)\,u_0^\varepsilon(x)\,dx\\
&=\int_\Omega w(x)u_0^\varepsilon(x)\,dx\\
&\ge \alpha\varepsilon .
\end{aligned}
\]
This proves the claim.
\end{proof}

\begin{theorem}[Uniformly elliptic extension, \texorpdfstring{$p>1$}{p>1}]\label{thm:infinite-lip-Wbp-ell}
We fix \(p>1\) and let \((P_t^{\mathcal L})_{t\ge0}\) be the homogeneous Dirichlet semigroup associated to \eqref{eq:ell-operator}. For every \(t>0\) and every \(m>0\), the density-induced map \(\rho\,dx\mapsto(P_t^{\mathcal L}\rho)\,dx\) satisfies
\[
\Lip_{W_{b,p}}\!\Bigl(\rho\,dx\mapsto(P_t^{\mathcal L}\rho)\,dx;
\{\rho\,dx:\rho\in L^1(\Omega),\ \rho\ge0,\ \int_\Omega \rho\,dx\le m\}\Bigr)=\infty.
\]
Moreover, the map
\[
\rho\,dx\longmapsto (P_t^{\mathcal L}\rho)\,dx
\]
is discontinuous at \(0\,dx\) on the cone
\[
\{\rho\,dx:\rho\in L^1(\Omega),\ \rho\ge0\}
\]
endowed with the metric \(W_{b,p}\).
\end{theorem}

\begin{proof}
By Proposition~\ref{prop:amplification-packet-uniform-ell}, for each sufficiently small \(\varepsilon>0\),
\[
\frac{W_{b,p}\bigl((P_t^{\mathcal L}u_0^\varepsilon)\,dx,0\bigr)}
{W_{b,p}(\mu_\varepsilon,0)}
\ge
\frac{\alpha^{1/p}}{2}\,\varepsilon^{-(1-1/p)}\xrightarrow[\varepsilon\downarrow0]{}+\infty.
\]
We fix \(m>0\) and set \(\sigma_\varepsilon:=m\mu_\varepsilon\). By Remark~\ref{rem:Wbp-homogeneity}, it follows that
\[
\frac{W_{b,p}\bigl((P_t^{\mathcal L}(m u_0^\varepsilon))\,dx,0\bigr)}
{W_{b,p}(\sigma_\varepsilon,0)}
=
\frac{W_{b,p}\bigl((P_t^{\mathcal L}u_0^\varepsilon)\,dx,0\bigr)}
{W_{b,p}(\mu_\varepsilon,0)}.
\]
Given that \(0=0\,dx\) belongs to the mass sublevel that we consider, the divergence of this ratio implies that the global \(W_{b,p}\)-Lipschitz constant on that set is infinite. For the discontinuity at \(0\), we set \(\tau_\varepsilon:=\varepsilon^{-1}\mu_\varepsilon\). Again, by Remark~\ref{rem:Wbp-homogeneity}, we obtain
\[
W_{b,p}(\tau_\varepsilon,0)^p
=\varepsilon^{-1}W_{b,p}(\mu_\varepsilon,0)^p
\le 2^p\varepsilon^{p-1}\xrightarrow[\varepsilon\downarrow0]{}0,
\]
whereas
\[
W_{b,p}\bigl((P_t^{\mathcal L}(\varepsilon^{-1}u_0^\varepsilon))\,dx,0\bigr)^p
=\varepsilon^{-1}W_{b,p}\bigl((P_t^{\mathcal L}u_0^\varepsilon)\,dx,0\bigr)^p
\ge \alpha.
\]
We conclude that the time-\(t\) map is discontinuous at \(0\,dx\).
\end{proof}

\subsection{Constant Dirichlet boundary values}

Let us fix \(c\ge0\). Due to the fact that the operator \(\mathcal L\) in \eqref{eq:ell-operator} does not contain a zero-order term, we have \(\mathcal L c=0\), and
the Dirichlet flow with boundary value \(c\) is the affine perturbation
\[
S_t^{\mathcal L}\rho_0:=c+P_t^{\mathcal L}(\rho_0-c).
\]
Whenever \(c\ge0\) and \(\rho_0\ge0\) on \(\Omega\), the parabolic maximum principle gives
\[
S_t^{\mathcal L}\rho_0\ge0 \qquad \text{in }\Omega \text{ for every } t\ge0.
\]
In particular, for such data, \((S_t^{\mathcal L}\rho_0)\,dx\in\mathcal M(\Omega)\).
For \(\mu=\rho\,dx\) with \(\rho\in L^1(\Omega)\), we write
\[
S_t^{\mathcal L}\mu:=(S_t^{\mathcal L}\rho)\,dx.
\]

\begin{theorem}[Affine Dirichlet case, \texorpdfstring{$p>1$}{p>1}]\label{thm:infinite-lip-Wbp-c-ell}
We fix $p>1$, $c\ge0$, and $t>0$. If we interpret the following object as a map from the class of initial data, we have
\[
\mathcal X_c\subset\mathcal M(\Omega)
\]
into \(\mathcal M(\Omega)\),
\[
\rho\,dx\longmapsto (S_t^{\mathcal L}\rho)\,dx
\]
has infinite \(W_{b,p}\)-Lipschitz constant:
\[
\Lip_{W_{b,p}}\!\left(S_t^{\mathcal L};\mathcal X_c\right)=\infty .
\]
More precisely, for every \(L>0\), there exist \(\rho_0^1,\rho_0^2\in\mathcal D_c\) with \(\rho_0^i\ge c\) on \(\overline\Omega\) and
\[
W_{b,p}\bigl(\rho_0^1\,dx,\rho_0^2\,dx\bigr)>0,
\]
such that the respective solutions satisfy
\[
\begin{multlined}
W_{b,p}\bigl(\rho^1(t,\cdot)\,dx,\rho^2(t,\cdot)\,dx\bigr)\\
\ge L\,W_{b,p}\bigl(\rho^1(0,\cdot)\,dx,\rho^2(0,\cdot)\,dx\bigr).
\end{multlined}
\]
In particular, \(S_t^{\mathcal L}:\mathcal X_c\to\mathcal M(\Omega)\) is
discontinuous at the point \(c\,dx\), where we endow both domain and target
with the metric induced by \(W_{b,p}\).
\end{theorem}

\begin{proof}
We fix $L>0$, and let $r_0>0$ and $\Pi:U\to\partial\Omega$ be as in Lemma~\ref{lem:tubular}. We also fix \(b_\ast\in\partial\Omega\) and extend \(\Pi\) to a Borel map
\[
\widehat\Pi:\Omega\to\partial\Omega
\]
by setting \(\widehat\Pi=\Pi\) on \(U\cap\Omega\) and \(\widehat\Pi=b_\ast\) on \(\Omega\setminus U\). Let $\alpha>0$, $r_*>0$, $u_0^\varepsilon\in C_c^\infty(\Omega)$, and $\mu_\varepsilon:=u_0^\varepsilon\,dx$ be as in Proposition~\ref{prop:amplification-packet-uniform-ell}, and choose $\varepsilon\in(0,\min\{r_*/4,r_0/4\})$. It follows that
\begin{equation}\label{eq:aux-onepoint-c}
\begin{aligned}
W_{b,p}(\mu_\varepsilon,0)&\le 2\varepsilon,\\
W_{b,p}\bigl((P_t^{\mathcal L}u_0^\varepsilon)\,dx,0\bigr)^p&\ge \alpha\,\varepsilon.
\end{aligned}
\end{equation}
We define
\[
\rho_0^2\equiv c,
\qquad
\rho_0^1:=c+M u_0^\varepsilon,
\]
with $M>0$ that we choose below. Because $\supp(u_0^\varepsilon)\subset\{\delta\ge\varepsilon\}$, both $\rho_0^1$ and $\rho_0^2$ equal $c$ on $\{\delta<\varepsilon\}$, hence in a neighbourhood of $\partial\Omega$. Moreover, $\rho_0^i\ge c\ge0$ holds. Let
\[
\rho^i(t,\cdot):=S_t^{\mathcal L}\rho_0^i,
\qquad i=1,2.
\]
At time \(0\), due to \(\supp(u_0^\varepsilon)\subset\{\delta\le2\varepsilon\}\subset U\), the plan
\[
\gamma_0:=(\mathrm{Id},\mathrm{Id})_\#(c\,dx)+(\mathrm{Id},\widehat\Pi)_\#(M u_0^\varepsilon\,dx)
\]
is admissible for \((\rho_0^1\,dx,\rho_0^2\,dx)\). Since \(|x-\Pi(x)|=\delta(x)\) on \(U\cap\Omega\), this gives
\[
\begin{aligned}
\int_{\overline\Omega\times\overline\Omega}|x-y|^p\,d\gamma_0(x,y)
&= M\int_\Omega \delta(x)^p\,u_0^\varepsilon(x)\,dx\\
&\le (2\varepsilon)^p M,
\end{aligned}
\]
so
\[
W_{b,p}(\rho_0^1\,dx,\rho_0^2\,dx)\le 2\,M^{1/p}\varepsilon.
\]
Let us set
\[
\begin{aligned}
E_c&:=\int_\Omega c\,\delta(x)^p\,dx,\\
B_\varepsilon&:=
W_{b,p}\bigl((P_t^{\mathcal L}u_0^\varepsilon)\,dx,0\bigr)^p.
\end{aligned}
\]
Given that $P_t^{\mathcal L}u_0^\varepsilon\ge0$, Lemma~\ref{lem:tozero} and the triangle inequality yield
\[
\begin{multlined}
W_{b,p}\bigl(\rho^1(t,\cdot)\,dx,\rho^2(t,\cdot)\,dx\bigr)\\
\ge W_{b,p}\bigl(\rho^1(t,\cdot)\,dx,0\bigr)
-W_{b,p}\bigl(\rho^2(t,\cdot)\,dx,0\bigr)\\
=(E_c+M B_\varepsilon)^{1/p}-E_c^{1/p}.
\end{multlined}
\]
By \eqref{eq:aux-onepoint-c}, we have $B_\varepsilon\ge\alpha\varepsilon$. We now set \(y:=M B_\varepsilon\) and consider
\[
F_y(s):=(s+y)^{1/p}-s^{1/p}\qquad(s\ge0).
\]
Notice that the function \(F_y\) is continuous on \([0,\infty)\) and is decreasing on
\((0,\infty)\). This holds because
\[
F_y'(s)=\frac1p\Bigl((s+y)^{\frac{1-p}{p}}-s^{\frac{1-p}{p}}\Bigr)\le0
\qquad(s>0).
\]
If $M\alpha\varepsilon\ge E_c$, then $y\ge E_c$, so
\[
F_y(E_c)\ge F_y(y)=(2^{1/p}-1)y^{1/p}.
\]
As a result,
\[
\begin{multlined}
W_{b,p}\bigl(\rho^1(t,\cdot)\,dx,\rho^2(t,\cdot)\,dx\bigr)\\
\ge F_y(E_c)
\ge F_y(y)
=(2y)^{1/p}-y^{1/p}
=(2^{1/p}-1)\,y^{1/p}\\
\ge (2^{1/p}-1)\,(M\alpha\varepsilon)^{1/p}.
\end{multlined}
\]
Therefore,
\[
\frac{W_{b,p}\bigl(\rho^1(t,\cdot)\,dx,\rho^2(t,\cdot)\,dx\bigr)}
{W_{b,p}\bigl(\rho^1(0,\cdot)\,dx,\rho^2(0,\cdot)\,dx\bigr)}
\ge \frac{(2^{1/p}-1)\alpha^{1/p}}{2}\,\varepsilon^{-(1-1/p)}.
\]
We choose $\varepsilon$ so that the right-hand side is at least $L$, and then choose $M>0$ with $M\ge E_c/(\alpha\varepsilon)$. This establishes the infinite-Lipschitz statement.

For the discontinuity at \(\rho\equiv c\), we simply repeat the construction with
\(M_0>0\) fixed and set
\[
M=M_0\varepsilon^{-1}.
\]
Since \(B_\varepsilon\ge\alpha\varepsilon\), this choice gives
\[
M B_\varepsilon\ge M_0\alpha.
\]
We obtain
\[
\begin{multlined}
W_{b,p}\bigl(\rho^1(t,\cdot)\,dx,\rho^2(t,\cdot)\,dx\bigr)\\
\ge (E_c+M B_\varepsilon)^{1/p}-E_c^{1/p}
\ge (E_c+M_0\alpha)^{1/p}-E_c^{1/p}>0.
\end{multlined}
\]
On the other hand,
\[
W_{b,p}(\rho_0^1\,dx,\rho_0^2\,dx)
\le 2M_0^{1/p}\varepsilon^{1-1/p}\xrightarrow[\varepsilon\downarrow0]{}0.
\]
We conclude that the time-\(t\) map is discontinuous at \(c\,dx\) as a map
\(\mathcal X_c\to\mathcal M(\Omega)\), where both spaces are endowed with
\(W_{b,p}\).
\end{proof}

\section{Results for the metric \texorpdfstring{$W_{b,2}$}{Wb,2}}\label{sec:fg-evi}

We now specialise the results to \(p=2\). The metric \(W_{b,2}\) is the finite-measure Figalli--Gigli boundary-reservoir distance. We consider the affine flow
\(S_t^c\rho=c+P_t(\rho-c)\) for arbitrary \(c\ge0\). Whenever we compare this flow with the entropy constructions of Ambrosio--Gigli and Figalli--Gigli, we fix the boundary value by the entropy normalisation: \(c=1\) for \(\int_\Omega(\rho\log\rho-\rho)\,dx\) \cite[Theorem~3.5]{FG10}, and \(c=e^{-1}\) for \(\int_\Omega\rho\log\rho\,dx\) \cite[Theorem~6.6]{AG13}. Indeed, the first variations of \(\int_\Omega(\rho\log\rho-\rho)\,dx\) and
\(\int_\Omega\rho\log\rho\,dx\) are, respectively, \(\log\rho\) and
\(\log\rho+1\), so the reservoir equilibrium condition selects the constant boundary values \(1\) and \(e^{-1}\).

In this section, we establish three results. First, we separate the literal orientation
given in \cite[Open Problem~6.7]{AG13} from the usual contractive
interpretation naturally suggested by EVI theory. We then show that both the literal
forward-nondecreasing inequality and the forward contractive interpretation fail
in the original finite-measure \(W_{b,2}\) metric, but they do so for different reasons. Finally, we describe the obstruction to any standard finite-\(\lambda\)
\(\mathrm{EVI}_\lambda\)-semigroup formulation, where the restriction to \(\mathcal X_c\) is the affine constant-boundary Dirichlet heat flow.

Below, a standard \(\mathrm{EVI}_\lambda\) semigroup means a metric-space EVI flow in the usual sense. We use the following standard fact: for each \(t>0\), there exists a finite constant \(C_{\lambda,t}\)
such that
\[
d(T_t x,T_t y)\le C_{\lambda,t}\,d(x,y)
\qquad\text{for all }x,y
\]
in the EVI domain. Following a common convention, we may take
\(C_{\lambda,t}=e^{-\lambda t}\). It follows that any such realisation on a
\(W_{b,2}\)-metric domain, which contains \(\mathcal X_c\) and agrees there with
the affine Dirichlet heat flow, would necessarily imply a finite fixed-time \(W_{b,2}\)-Lipschitz bound for \(S_t^c:\mathcal X_c\to\mathcal M(\Omega)\).

We consider two distinct monotonicity questions. First, we note that the inequality from \cite[Open Problem~6.7]{AG13} has the forward-nondecreasing orientation
\[
   W_{b,2}\bigl(\rho^1(s,\cdot)\,dx,\rho^2(s,\cdot)\,dx\bigr)
   \le
   W_{b,2}\bigl(\rho^1(t,\cdot)\,dx,\rho^2(t,\cdot)\,dx\bigr),
   \qquad t>s.
\]
In Remark~\ref{rem:literal-AG-direction}, we disprove this orientation. We note that the usual EVI/contractive interpretation would instead require a finite forward Lipschitz estimate. In Corollaries~\ref{cor:p2-positive-time} and~\ref{cor:noEVI}, we rule out that interpretation in the original finite-measure \(W_{b,2}\) metric.

\begin{lemma}[A total-variation bound for \texorpdfstring{$W_{b,2}$}{Wb,2}]\label{lem:Wb2-TV}
For all finite nonnegative measures \(\mu,\nu\in\mathcal M(\Omega)\),
\[
W_{b,2}(\mu,\nu)^2\le D_\Omega^2\,\|\mu-\nu\|_{\mathrm{TV}}.
\]
\end{lemma}

\begin{proof}
We write \(\mu=f\,\lambda\), \(\nu=g\,\lambda\) with \(\lambda:=\mu+\nu\), and set
\[
m:=\min\{f,g\}\,\lambda.
\]
We further fix \(b_0\in\partial\Omega\) and define
\[
\begin{aligned}
\gamma:=\,&(\mathrm{Id},\mathrm{Id})_\# m
+(\mathrm{Id},b_0)_\#(\mu-m)\\
&+(b_0,\mathrm{Id})_\#(\nu-m).
\end{aligned}
\]
This implies \(\gamma\in\Adm(\mu,\nu)\). Moreover,
\[
\int_{\overline\Omega\times\overline\Omega}|x-y|^2\,d\gamma(x,y)
\le D_\Omega^2\bigl((\mu-m)(\Omega)+(\nu-m)(\Omega)\bigr).
\]
Since
\[
(\mu-m)(\Omega)+(\nu-m)(\Omega)=\|\mu-\nu\|_{\mathrm{TV}},
\]
it suffices to take the infimum over admissible plans to obtain the claim.
\end{proof}

\begin{remark}[The orientation in Ambrosio--Gigli's Open Problem~6.7]\label{rem:literal-AG-direction}
The inequality given in \cite[Open Problem~6.7]{AG13}, interpreted with the
orientation written there, already fails for smooth data. More precisely, we fix \(c\ge0\), and, in particular, \(c=e^{-1}\). We choose
\(\psi\in C_c^\infty(\Omega)\), \(\psi\ge0\), \(\psi\not\equiv0\), and set
\[
   \rho_0^1:=c+\psi,
   \qquad
   \rho_0^2:=c.
\]
It follows that \(\rho_0^1,\rho_0^2\in\mathcal D_c\). Let
\[
   \rho^1(t,\cdot):=c+P_t\psi,
   \qquad
   \rho^2(t,\cdot):=c.
\]
These are classical solutions of the Dirichlet heat equation with boundary
value \(c\). We define
\[
   d(t):=
   W_{b,2}\bigl(\rho^1(t,\cdot)\,dx,\rho^2(t,\cdot)\,dx\bigr)
   =
   W_{b,2}\bigl((c+P_t\psi)\,dx,c\,dx\bigr).
\]
For every \(s>0\), we have \(P_s\psi\not\equiv0\), so the two measures are
distinct and, therefore, \(d(s)>0\), due to the fact that \(W_{b,2}\) is a metric.

On the other hand, \(d(t)\to0\) as \(t\to\infty\). Indeed, by
Lemma~\ref{lem:Wb2-TV}, it holds that
\[
   d(t)^2
   \le
   D_\Omega^2\,\|P_t\psi\|_{L^1(\Omega)}.
\]
Let \(\lambda_1>0\) be the first Dirichlet eigenvalue of \(-\Delta\) on
\(\Omega\). By the Dirichlet spectral gap,
\[
   \|P_t\psi\|_{L^1(\Omega)}
   \le
   |\Omega|^{1/2}\|P_t\psi\|_{L^2(\Omega)}
   \le
   |\Omega|^{1/2}e^{-\lambda_1 t}\|\psi\|_{L^2(\Omega)}
   \longrightarrow0 .
\]
It follows that \(d(t)\to0\). If we fix an arbitrary \(s>0\) and then choose \(t>s\) large enough, we obtain
\[
   W_{b,2}\bigl(\rho^1(s,\cdot)\,dx,\rho^2(s,\cdot)\,dx\bigr)
   >
   W_{b,2}\bigl(\rho^1(t,\cdot)\,dx,\rho^2(t,\cdot)\,dx\bigr),
\]
which disproves the forward nondecreasing inequality
\[
   W_{b,2}\bigl(\rho^1(s,\cdot)\,dx,\rho^2(s,\cdot)\,dx\bigr)
   \le
   W_{b,2}\bigl(\rho^1(t,\cdot)\,dx,\rho^2(t,\cdot)\,dx\bigr),
   \qquad t>s,
\]
in the given orientation.
\end{remark}

\begin{corollary}[Amplification on \(\mathcal D_c\)]\label{cor:p2-twopoint}
Let us fix $c\ge0$. For every \(L>0\) and every \(t>0\), there exist
\(\rho_0^1,\rho_0^2\in \mathcal D_c\), with \(\rho_0^i\ge c\) on \(\overline\Omega\) and
\[
W_{b,2}\bigl(\rho_0^1\,dx,\rho_0^2\,dx\bigr)>0,
\]
such that the respective solutions satisfy
\[
W_{b,2}\bigl((S_t\rho_0^1)\,dx,(S_t\rho_0^2)\,dx\bigr)
\ge L\,W_{b,2}\bigl(\rho_0^1\,dx,\rho_0^2\,dx\bigr).
\]
In particular, for every \(t>0\), the affine Dirichlet time-\(t\) map \(S_t:\mathcal X_c\to\mathcal M(\Omega)\) is not contractive in \(W_{b,2}\).
\end{corollary}

\begin{proof}
It suffices to apply Theorem~\ref{thm:infinite-lip-Wbp-c-ell} with \(\mathcal L=\Delta\) and \(p=2\), and to take \(L>1\) to show that contractivity fails.
\end{proof}

\begin{corollary}[No forward nonincreasing monotonicity along positive-time trajectories]
\label{cor:p2-positive-time}
We fix \(c\ge0\). There exist classical solutions \(\rho^1,\rho^2\) of the Dirichlet heat
equation in \(\Omega\) with boundary value \(c\), and times \(0<s<t\), such that
\[
W_{b,2}\bigl(\rho^1(t,\cdot)\,dx,\rho^2(t,\cdot)\,dx\bigr)
>
W_{b,2}\bigl(\rho^1(s,\cdot)\,dx,\rho^2(s,\cdot)\,dx\bigr).
\]
It follows that the map \(r\mapsto W_{b,2}\bigl(\rho^1(r,\cdot)\,dx,\rho^2(r,\cdot)\,dx\bigr)\) need not be nonincreasing on \((0,\infty)\). For the boundary value \(c=e^{-1}\), this gives the negative example to the contractive, forward-nonincreasing interpretation of the Ambrosio--Gigli question. Corollary~\ref{cor:noEVI} gives the stronger exclusion of every standard \(\mathrm{EVI}_\lambda\)-semigroup interpretation with finite \(\lambda\).
\end{corollary}

\begin{proof}
We fix \(\tau>0\) and choose \(L>1\). By Corollary~\ref{cor:p2-twopoint}, there exist
\(\rho_0^1,\rho_0^2\in\mathcal D_c\) with
\[
W_{b,2}\bigl(\rho_0^1\,dx,\rho_0^2\,dx\bigr)>0
\]
such that, if we write,
\[
\rho^i(r,\cdot):=S_r\rho_0^i,
\qquad
d(r):=W_{b,2}\bigl(\rho^1(r,\cdot)\,dx,\rho^2(r,\cdot)\,dx\bigr),
\]
we obtain
\[
d(\tau)\ge L\,d(0)>d(0).
\]
We claim that \(d(r)\to d(0)\) as \(r\downarrow0\). Indeed, by the triangle inequality and Lemma~\ref{lem:Wb2-TV},
\[
\begin{aligned}
|d(r)-d(0)|
&\le W_{b,2}\bigl(\rho^1(r,\cdot)\,dx,\rho_0^1\,dx\bigr)
   +W_{b,2}\bigl(\rho^2(r,\cdot)\,dx,\rho_0^2\,dx\bigr)\\
&\le D_\Omega\sum_{i=1}^2
   \|\rho^i(r,\cdot)-\rho_0^i\|_{L^1(\Omega)}^{1/2}.
\end{aligned}
\]
Because
\[
\rho^i(r,\cdot)=c+P_r(\rho_0^i-c)
\]
and \(P_r\) is strongly continuous on \(L^1(\Omega)\), the right-hand side tends to
\(0\) as \(r\downarrow0\), which implies \(d(r)\to d(0)\). Since \(d(\tau)>d(0)\), we may choose \(s\in(0,\tau)\) so small that \(d(s)<d(\tau)\). The claim follows by setting \(t:=\tau\).
\end{proof}

We now discuss the exclusion of standard \(\mathrm{EVI}_\lambda\)-semigroup realisations. In the argument, we use the fixed-time Lipschitz result given at the beginning of this section. In the following statement, \(D(\mathcal E)\) is endowed with the metric obtained by the restriction of \(W_{b,2}\) to \(D(\mathcal E)\).

\begin{corollary}[No standard \(\mathrm{EVI}_\lambda\)-semigroup interpretation on the finite Figalli--Gigli subspace]\label{cor:noEVI}
We fix \(c\ge0\) and \(\lambda\in\mathbb R\). There does not exist a functional
\(\mathcal E\) and a semigroup \(T_t:D(\mathcal E)\to D(\mathcal E)\) that would satisfy the standard metric-space \(\mathrm{EVI}_\lambda\) property on \(D(\mathcal E)\), endowed with the restricted \(W_{b,2}\)-metric, and also satisfy the fixed-time Lipschitz result, with
\[
\{\rho\,dx:\rho\in\mathcal D_c\}\subset D(\mathcal E),
\]
such that
\[
T_t(\rho\,dx)=(S_t\rho)\,dx
\qquad \forall \rho\in\mathcal D_c,\ \forall t\ge0 .
\]
\end{corollary}

Notice that for \(c\) equal to the entropy-selected boundary constant, this gives the obstruction to a standard \(\mathrm{EVI}_\lambda\) interpretation of the Figalli--Gigli heat flow. For other \(c\), it is an obstruction for the affine constant-boundary heat flow \(S_t^c\).

\begin{proof}
Indeed, an \(\mathrm{EVI}_\lambda\) flow satisfies a finite fixed-time
Lipschitz estimate. For every \(t>0\), there exists \(C_{\lambda,t}<\infty\) such
that
\[
W_{b,2}(T_t x,T_t y)\le C_{\lambda,t}W_{b,2}(x,y)
\qquad\forall x,y\in D(\mathcal E).
\]
This is the standard EVI contraction estimate (see
\cite[Chapter~4]{AGS08}, and \cite[Proposition~3.1]{San17} for an
overview). For each \(t_0>0\) and every
\(\rho,\sigma\in\mathcal D_c\),
\[
\begin{multlined}
W_{b,2}\bigl((S_{t_0}\rho)\,dx,(S_{t_0}\sigma)\,dx\bigr)
=
W_{b,2}\bigl(T_{t_0}(\rho\,dx),T_{t_0}(\sigma\,dx)\bigr)\\
\le
C_{\lambda,t_0}\,W_{b,2}(\rho\,dx,\sigma\,dx).
\end{multlined}
\]
We choose \(L>C_{\lambda,t_0}\). By Corollary~\ref{cor:p2-twopoint}, we have
\(\rho,\sigma\in\mathcal D_c\) for which
\[
W_{b,2}\bigl((S_{t_0}\rho)\,dx,(S_{t_0}\sigma)\,dx\bigr)
\ge L\,W_{b,2}(\rho\,dx,\sigma\,dx),
\]
which contradicts the previous \(C_{\lambda,t_0}\)-Lipschitz estimate.
\end{proof}

\section{Conclusion}
The sharp modulus theorem shows that the boundary reservoir has two distinct
fixed-time regimes: a Lipschitz endpoint at \(p=1\), determined by
Kantorovich--Rubinstein duality, and a genuinely nonlinear \(1/p\)-H\"older
regime for every \(p>1\), governed by Hopf amplification of boundary layers. In
the quadratic Figalli--Gigli case, this identifies the obstruction to EVI-type contractivity in the finite-measure metric.

\appendix

\section{Proof of the Hopf lower bound}\label{app:hopf-detail}
In this appendix, we analyse in more detail the barrier argument from Lemma~\ref{lem:hopf-linear-general}. The argument is the parabolic Hopf boundary point lemma for a uniformly parabolic operator with smooth coefficients. Note that we allow the operator to be nonsymmetric and we use the inward normal convention. For regularity, we apply the global \(W^{2,1}_r\) Dirichlet estimate on \(C^2\) cylinders for linear uniformly parabolic equations with smooth coefficients and bounded lower-order terms. We rewrite
\[
\nabla\!\cdot(A\nabla u)+b\cdot\nabla u+qu
=
\sum_{i,j}a_{ij}\partial_{ij}u
+\sum_i\Bigl(\sum_j\partial_j a_{ji}+b_i\Bigr)\partial_i u
+qu.
\]
We may now apply the standard global estimates and the maximum principle for uniformly parabolic operators. We could also invoke the classical parabolic Hopf boundary lemma for \(e^{-Kt}u\) after choosing \(K\ge \|q\|_\infty\).

\begin{proof}[Proof of Lemma~\ref{lem:hopf-linear-general}]
Let us fix a connected component $\Omega'\subset\Omega$ with $\eta\not\equiv0$ on $\Omega'$, and write $\mathbf n$ for the inward unit normal on $\partial\Omega'$. We first obtain positive-time \(W_r^{2,1}\)-regularity and thus
\(C^1\)-regularity up to the boundary at time \(t_0\). We then apply the parabolic strong maximum principle and a common barrier argument to get a strict inward normal derivative on \(\partial\Omega'\). The latter integrates along the tubular normal coordinates into the linear lower bound we are considering.

Let us fix \(\tau\in(0,t_0/2)\), and choose \(\chi\in C_c^\infty((t_0-\tau,t_0+\tau))\)
such that \(0\le\chi\le1\) and \(\chi\equiv1\) on
\([t_0-\tau/2,t_0+\tau/2]\). Because \(\eta\in L^\infty(\Omega')\) holds, the maximum
principle gives \(w\in L^\infty((0,t_0+\tau)\times\Omega')\). We set
\[
u(t,x):=\chi(t)w(t,x).
\]
It follows that \(u\) is a distributional solution on \((t_0-\tau,t_0+\tau)\times\Omega'\) of the corresponding homogeneous Dirichlet problem with a zero initial datum at \(t=t_0-\tau\) and a right-hand side of \((\partial_t\chi)w\). Given that \(w\in L^\infty\), this right-hand side is in \(L^r((t_0-\tau,t_0+\tau)\times\Omega')\) for every \(1<r<\infty\). By Lemma~\ref{lem:positive-time-regularity}, which we apply to this cutoff solution, we obtain
\[
u\in W_r^{2,1}\bigl((t_0-\tau,t_0+\tau)\times\Omega'\bigr)
\qquad\text{for every }1<r<\infty.
\]
By the standard energy or duality uniqueness for the zero-initial Cauchy--Dirichlet problem in \(L^r\), this \(W^{2,1}_r\)-representative coincides with the semigroup solution. In particular, the same holds on the smaller slab \((t_0-\tau/2,t_0+\tau/2)\times\Omega'\). Let us choose \(r>n+2\). By the parabolic Morrey embedding in the \(W^{2,1}_r\)-scale, we obtain
\[
\nabla_x u\in C^{\alpha,\alpha/2}
\bigl([t_0-\tau/2,t_0+\tau/2]\times\overline{\Omega'}\bigr)
\]
for some \(\alpha\in(0,1)\). In other terms, \(u\) is spatially \(C^{1+\alpha}\) and temporally Hölder continuous on the smaller slab. In particular, for every \(t\in[t_0-\tau/2,t_0+\tau/2]\), the slice \(u(t,\cdot)\) belongs to \(C^1(\overline{\Omega'})\), and the map \(t\mapsto \nabla_x u(t,\cdot)\) is continuous with values in
\(C(\overline{\Omega'})\). Given that \(u=w\) on \([t_0-\tau/2,t_0+\tau/2]\times\Omega'\), the same conclusion holds for \(w\). It follows that \(w(t_0,\cdot)\in C^1(\overline{\Omega'})\). We set
\[
K:=\max\{0,\|q\|_{L^\infty(\Omega)}\},
\qquad
\widetilde w(t,x):=e^{-Kt}w(t,x).
\]
This implies that \(\widetilde w\) solves
\[
\partial_t\widetilde w
=\nabla\cdot(A\nabla\widetilde w)+b\cdot\nabla\widetilde w+(q-K)\widetilde w,
\qquad
\widetilde w|_{\partial\Omega'}=0,
\]
with a zero-order coefficient \(q-K\le0\). Let us now rewrite the principal part in nondivergence form,
\[
\partial_t\widetilde w
=\sum_{i,j=1}^n a_{ij}(x)\partial_{ij}\widetilde w
+\sum_{i=1}^n \widetilde b_i(x)\partial_i\widetilde w
+(q-K)\widetilde w,
\qquad \widetilde b_i:=\sum_{j=1}^n \partial_j a_{ji}+b_i.
\]
Because \(D^2\widetilde w\) is symmetric,
\[
\sum_{i,j=1}^n a_{ij}\partial_{ij}\widetilde w
=
\sum_{i,j=1}^n (A_s)_{ij}\partial_{ij}\widetilde w .
\]
The above implies that we may take the principal coefficient in the nondivergence operator to be the uniformly elliptic symmetric matrix \(A_s\). It follows that \(\widetilde w\) is a strong solution of a uniformly parabolic nondivergence equation on \((t_0-\tau/2,t_0+\tau/2)\times\Omega'\) with smooth bounded coefficients and a zero-order term \(q-K\le 0\). By the previously shown \(W_r^{2,1}\)-regularity (for every \(r<\infty\)) and the choice of \(r>n+2\), we obtain
\[
\nabla_x\widetilde w\in
C\bigl([t_0-\tau/2,t_0+\tau/2]\times\overline{\Omega'}\bigr).
\]
We set
\[
L=-\partial_t+\sum_{i,j}a_{ij}\partial_{ij}
+\sum_i\widetilde b_i\partial_i+(q-K).
\]
On every cylinder \((0,T]\times\Omega'\), we apply the forward maximum principle to \(h=-\widetilde w\). For the initial datum, we have \(h(0,\cdot)=-\eta\le0\). Together with the zero-order coefficient \(q-K\le0\) and the zero lateral boundary value, we obtain
\[
\widetilde w\ge0
\qquad\text{on }[0,T]\times\Omega' .
\]
Since \(\eta\ge0\) is not identically zero on the connected component
\(\Omega'\), by the parabolic strong maximum principle, we obtain
\[
\widetilde w>0
\qquad\text{in }(0,T]\times\Omega'
\]
for every \(T>0\). If we take \(T=t_0+\tau\), we notice that \(-\widetilde w\) attains a strict boundary maximum \(0\) at each \((t_0,\xi)\in (t_0-\tau,t_0+\tau)\times\partial\Omega'\). Let us fix \(\xi\in\partial\Omega'\). Because \(\partial\Omega'\) is \(C^2\), the interior ball condition holds at \(\xi\). More precisely, there exist \(R>0\) and \(y=\xi+R\mathbf n(\xi)\in\Omega'\) such that
\[
B_R(y)\subset\Omega',
\qquad
\overline{B_R(y)}\cap\partial\Omega'=\{\xi\}.
\]
We may assume \(R^2<\tau/2\) after we shrink \(R\), if necessary. Next, we set
\[
PF:=\bigl\{(t,x): |x-y|^2+(t_0-t)<R^2,\ t<t_0\bigr\}
\subset (t_0-\tau/2,t_0)\times\Omega'
\]
and define \(v:=-\widetilde w\). It follows that
\[
Lv:=-\partial_t v+\sum_{i,j=1}^n a_{ij}(x)\partial_{ij}v
+\sum_{i=1}^n \widetilde b_i(x)\partial_i v +(q-K)v =0
\quad\text{a.e. in }PF,
\]
with uniformly parabolic bounded coefficients and a zeroth-order term
\(q-K\le 0\). Moreover,
\[
v(t_0,\xi)=0,
\qquad
v<0\ \text{in }PF,
\]
and by regularity we established above, we have
\[
v\in W_r^{2,1}(PF)\cap C(\overline{PF}),
\qquad
\nabla_x v\in C(\overline{PF})
\]
for \(r>n+2\).

Let
\[
\zeta(t,x):=e^{-\alpha(|x-y|^2+(t_0-t))}-e^{-\alpha R^2}
\]
and
\[
P:=PF\cap\{|x-y|>R/2\}.
\]
Let \(0<\lambda\le\Lambda\) denote ellipticity bounds for \(A\) on
\(\overline{\Omega'}\). For \(z:=x-y\), by direct computation, we obtain
\[
-\partial_t\zeta=-\alpha e^{-\alpha(|z|^2+(t_0-t))},
\qquad
\partial_i\zeta=-2\alpha z_i e^{-\alpha(|z|^2+(t_0-t))},
\]
and
\[
\partial_{ij}\zeta=
\bigl(-2\alpha\delta_{ij}+4\alpha^2 z_i z_j\bigr)
e^{-\alpha(|z|^2+(t_0-t))}.
\]
It follows that
\[
\begin{aligned}
L\zeta
&=
e^{-\alpha(|z|^2+(t_0-t))}
\Bigl(
4\alpha^2\langle A_s(x)z,z\rangle
-\alpha
-2\alpha\,\mathrm{tr}\,A(x)
-2\alpha\,\widetilde b(x)\cdot z
\Bigr)\\
&\quad +(q(x)-K)\zeta.
\end{aligned}
\]
On \(P\), we have \(|z|>R/2\) and \(|z|<R\), so
\[
\langle A_s(x)z,z\rangle\ge \lambda |z|^2\ge \lambda R^2/4,
\qquad
|z|\le R,
\qquad
0\le \zeta\le e^{-\alpha(|z|^2+(t_0-t))}.
\]
Therefore,
\[
L\zeta
\ge
e^{-\alpha(|z|^2+(t_0-t))}
\Bigl(
\lambda R^2\alpha^2
-\alpha\bigl(1+2n\Lambda+2R\|\widetilde b\|_{L^\infty(\Omega')}\bigr)
-\|q-K\|_{L^\infty(\Omega')}
\Bigr).
\]
We now choose \(\alpha>0\) sufficiently large, so that it depends only on the coefficient bounds, the ellipticity constants, and \(R\). This yields
\[
L\zeta>0
\qquad\text{in }P.
\]
We apply the forward maximum principle in the following alternative form. To this end, set
\[
\mathscr P:=-L
=\partial_t-\sum_{i,j}a_{ij}\partial_{ij}
-\sum_i\widetilde b_i\partial_i+(K-q).
\]
Since \(K-q\ge0\), the standard parabolic maximum principle for \(\mathscr P\)
states that if \(\mathscr P h\le0\) in \(P\) and \(h\le0\) on the lower-time/lateral parabolic boundary of \(P\), then \(h\le0\) in \(P\). In other words, if \(Lh\ge0\) in \(P\), then the same conclusion holds. Note that the terminal face \(\overline P\cap\{t=t_0\}\) is not part of the parabolic boundary we prescribe. Let us now apply this maximum principle. Since \(Lv=0\) and \(L\zeta>0\) in \(P\), we have
\[
L(v+\varepsilon\zeta)>0
\qquad\text{in }P.
\]
The relevant parabolic boundary for this operator is the initial/lateral boundary of \(P\). Moreover, it lies in
\[
\Sigma_{\rm out}:=\overline P\cap
\bigl\{|x-y|^2+(t_0-t)=R^2\bigr\}
\]
and
\[
\Sigma_{\rm in}:=\overline P\cap\{|x-y|=R/2\}.
\]
On \(\Sigma_{\rm out}\), we have \(\zeta=0\) and \(v\le0\). Equality is possible
only at the tangency point \((t_0,\xi)\), which lies on the terminal face. On
\(\Sigma_{\rm in}\), by compactness and strict positivity of \(\widetilde w\), we have
\(v=-\widetilde w\le -m\) for some \(m>0\). If we choose \(\varepsilon>0\)
sufficiently small, we obtain
\[
v+\varepsilon\zeta\le 0
\qquad\text{on the parabolic boundary of }P.
\]
By the maximum principle,
\[
v+\varepsilon\zeta\le 0
\qquad\text{in }P.
\]
By continuity, this inequality extends to \(\overline P\), which includes the terminal face \(t=t_0\). It follows that for every \(s\in(0,R/2)\),
\[
\frac{(v+\varepsilon\zeta)(t_0,\xi+s\mathbf n(\xi))
-(v+\varepsilon\zeta)(t_0,\xi)}{s}\le 0.
\]
We now pass to the limit \(s\downarrow0\) and apply the continuity of \(\nabla_x v\), which gives
\[
\partial_{\mathbf n}v(t_0,\xi)\le
-\varepsilon\,\partial_{\mathbf n}\zeta(t_0,\xi)<0.
\]
Therefore,
\[
\partial_{\mathbf n}\widetilde w(t_0,\xi)>0.
\]
Recall that $\xi\in\partial\Omega'$ was arbitrary, so the above holds for every
\(\xi\in\partial\Omega'\). We further have
\[
\partial_{\mathbf n}\widetilde w(t_0,\xi)>0
\qquad\forall \xi\in\partial\Omega',
\]
so
\[
\partial_{\mathbf n}w(t_0,\xi)
=e^{Kt_0}\partial_{\mathbf n}\widetilde w(t_0,\xi)>0
\qquad\forall \xi\in\partial\Omega'.
\]

By continuity of $\partial_{\mathbf n}w(t_0,\cdot)$ on compact $\partial\Omega'$, there exists $m>0$ such that
\[
\partial_{\mathbf n}w(t_0,\xi)\ge m\qquad\forall \xi\in\partial\Omega'.
\]
We now apply Lemma~\ref{lem:tubular} to the bounded \(C^2\) domain \(\Omega'\).
Then there exists \(r_0>0\) such that the tubular map
\[
F:\partial\Omega'\times[0,r_0]\to
\{x\in\overline{\Omega'}:\delta(x)\le r_0\},
\qquad
F(\xi,s):=\xi+s\,\mathbf n(\xi),
\]
is well defined and satisfies \(\delta(F(\xi,s))=s\). Since \(w(t_0,\cdot)\in C^1(\overline{\Omega'})\), the map \(x\mapsto \nabla w(t_0,x)\) is continuous on \(\overline{\Omega'}\). It holds that
\[
g(\xi,s):=\nabla w(t_0,F(\xi,s))\cdot\mathbf n(\xi)
\]
is continuous on \(\partial\Omega'\times[0,r_0]\). Because
\[
g(\xi,0)=\partial_{\mathbf n}w(t_0,\xi)\ge m
\qquad\forall \xi\in\partial\Omega',
\]
by compactness, we obtain \(r\in(0,r_0]\) such that
\[
g(\xi,s)\ge \frac m2
\qquad\forall (\xi,s)\in\partial\Omega'\times[0,r].
\]

We fix \(\xi\in\partial\Omega'\). Since \(\partial_sF(\xi,s)=\mathbf n(\xi)\)
and \(w(t_0,\xi)=0\) on \(\partial\Omega'\), by the chain rule, we have for every
\(s\in[0,r]\),
\[
\frac{d}{ds}w(t_0,F(\xi,s))
=
\nabla w(t_0,F(\xi,s))\cdot \mathbf n(\xi)
=
g(\xi,s).
\]
Therefore,
\[
w(t_0,F(\xi,s))
=
\int_0^s g(\xi,\tau)\,d\tau
\ge \frac m2\,s
\qquad\forall s\in[0,r].
\]
Because \(\delta(F(\xi,s))=s\), it follows that
\[
w(t_0,x)\ge \frac m2\,\delta(x)
\qquad\text{for all }x\in\Omega'\text{ with }\delta(x)\le r.
\]
This establishes the estimate with $\alpha:=m/2$. If \(\eta\equiv0\) on a connected component \(\Omega''\), by uniqueness for the homogeneous Dirichlet problem, we have \(w(t,\cdot)\equiv0\) on \(\Omega''\) for all \(t\ge0\). For each component, it suffices to apply the same regularity argument to obtain the \(C^1\)-extension property.
\end{proof}

\section{Shortcut representation of \texorpdfstring{$W_{b,1}$}{Wb,1}}\label{app:Wb1-shortcut}
In this part, we analyse the collapsed-boundary representation from Section~\ref{sec:Wb1}. The construction is the \(p=1\) collapsed-boundary
specialisation of the shortcut/relative-transport interpretation. In particular,
Bate--Garc\'ia Pulido \cite{BGP24} identify the Figalli--Gigli partial-transport geometry with ordinary Wasserstein transport on the one-point
shortcut completion endowed with
\[
d_*(x,y)=\min\{|x-y|,\delta(x)+\delta(y)\},
\]
while Bubenik--Elchesen \cite{BE24} prove the corresponding relative
Kantorovich--Rubinstein duality in the general metric-pair framework. The persistence-diagram antecedent is due to Divol--Lacombe \cite{DL21} (see also \cite[Remark~3.4]{BGP24} for the relation with the partial-transport metric).

\begin{lemma}[The collapsed-boundary metric]\label{lem:dstar-compact}
Let $X_*:=\Omega\sqcup\{\ast\}$ and define $d_*$ by
\[
\begin{aligned}
d_*(x,y)&:=\min\{|x-y|,\delta(x)+\delta(y)\}
&& (x,y\in\Omega),\\
d_*(x,\ast)&:=\delta(x),\\
d_*(\ast,\ast)&:=0.
\end{aligned}
\]
It follows that $d_*$ is a metric on $X_*$, and $(X_*,d_*)$ is compact.
\end{lemma}

\begin{proof}
We consider the complete graph with vertex set $X_*$ and edge lengths
\[
\begin{aligned}
\ell(x,y)&:=|x-y| && (x,y\in\Omega),\\
\ell(x,\ast)&:=\delta(x),\\
\ell(\ast,\ast)&:=0.
\end{aligned}
\]
Let $d_G$ be the associated shortest-path metric. For $x,y\in\Omega$, any path from $x$ to $y$ that does not pass through $\ast$ has length at least
$|x-y|$ by the Euclidean triangle inequality, and an arbitrary path from $x$ to $y$ that passes through $\ast$ has length at least $\delta(x)+\delta(y)$. To estimate paths from $x$ to $\ast$, it suffices to consider paths where $\ast$ appears only at the final vertex. If a path reaches $\ast$ earlier, then the deletion of the remaining loop does not increase its length. Let
\[
x=x_0,x_1,\dots,x_{m-1}\in\Omega,\qquad x_m=\ast.
\]
If we write $\delta(\ast):=0$ and use that $\delta$ is $1$-Lipschitz on $\Omega$, we obtain
\[
\begin{aligned}
\sum_{i=1}^{m}\ell(x_{i-1},x_i)
&=\sum_{i=1}^{m-1}|x_{i-1}-x_i|+\delta(x_{m-1})\\
&\ge \sum_{i=1}^{m-1}|\delta(x_{i-1})-\delta(x_i)|+\delta(x_{m-1})\\
&\ge \delta(x_0)=\delta(x).
\end{aligned}
\]
We infer that every path from $x$ to $\ast$ has length at least $\delta(x)$, and analogously, every path from $\ast$ to $y$ has length at least $\delta(y)$.
On the other hand, the one-edge path $x\to y$ has length $|x-y|$, and the two-edge path $x\to\ast\to y$ has length $\delta(x)+\delta(y)$. We conclude that
\[
d_G(x,y)=\min\{|x-y|,\delta(x)+\delta(y)\}=d_*(x,y).
\]
It is also true that $d_G(x,\ast)=\delta(x)=d_*(x,\ast)$, so $d_*=d_G$. Therefore, $d_*$ is a metric.

To prove compactness, let $(u_k)$ be an arbitrary sequence in $X_*$. If $u_k=\ast$ for infinitely many $k$, then a constant subsequence converges. Otherwise, after passing to a subsequence, $u_k=x_k\in\Omega$ for all $k$. Given that $\overline\Omega$ is compact, after passing to a further subsequence, we may assume $x_k\to \bar x\in\overline\Omega$. If $\bar x\in\Omega$, then
\[
d_*(x_k,\bar x)\le |x_k-\bar x|\to0.
\]
If $\bar x\in\partial\Omega$, then
\[
d_*(x_k,\ast)=\delta(x_k)\le |x_k-\bar x|\to0.
\]
It follows that every sequence in $X_*$ admits a convergent subsequence, so $(X_*,d_*)$ is compact.
\end{proof}

\begin{proposition}[Reduction of \texorpdfstring{$W_{b,1}$}{Wb,1} to a standard \texorpdfstring{$W_1$}{W1} problem]\label{prop:Wb1-augmented}
Let \((X_*,d_*)\) be as in Lemma~\ref{lem:dstar-compact}.
For $\mu,\nu\in\mathcal M(\Omega)$, we define
\[
\widetilde\mu:=\mu+\nu(\Omega)\,\delta_\ast,
\qquad
\widetilde\nu:=\nu+\mu(\Omega)\,\delta_\ast.
\]
It follows that
\[
W_{b,1}(\mu,\nu)=W_{1,d_*}(\widetilde\mu,\widetilde\nu),
\]
where \(W_{1,d_*}\) denotes the \(1\)-Wasserstein distance for finite Borel measures of equal total mass.
\end{proposition}

\begin{proof}
Notice that the topology induced by \(d_*\) on \(\Omega\) agrees with the Euclidean topology. If \(x\in\Omega\) and \(0<r<\delta(x)\), then
\[
B_{d_*}(x,r)=B^{\mathrm{eucl}}(x,r).
\]
The Borel \(\sigma\)-algebra on \(\Omega\) inherited from \(X_*\) is the usual
Euclidean Borel \(\sigma\)-algebra, so we may interpret \(\mu\) and \(\nu\) canonically as finite Borel measures on \(X_*\) supported in \(\Omega\).

We first prove
\[
W_{1,d_*}(\widetilde\mu,\widetilde\nu)\le W_{b,1}(\mu,\nu).
\]
To this end, fix $\gamma\in\Adm(\mu,\nu)$ and, by Remark~\ref{rem:discard-bdrybdry}, we may discard boundary--boundary mass. We write
\[
\begin{aligned}
\gamma_{II}&:=\gamma\restriction_{\Omega\times\Omega},\\
\gamma_{IB}&:=\gamma\restriction_{\Omega\times\partial\Omega},\\
\gamma_{BI}&:=\gamma\restriction_{\partial\Omega\times\Omega}.
\end{aligned}
\]
Let $\iota:\Omega\hookrightarrow X_*$ be the inclusion map and define a transport plan $\pi$ on $X_*\times X_*$ by
\[
\begin{aligned}
\pi:=\,&(\iota,\iota)_\#\gamma_{II}
+(\iota,\ast)_\#\gamma_{IB}\\
&+(\ast,\iota)_\#\gamma_{BI}
+\gamma_{II}(\Omega\times\Omega)\,\delta_{(\ast,\ast)}.
\end{aligned}
\]
It is straightforward to check directly that \(\pi\) is a coupling of \(\widetilde\mu\) and \(\widetilde\nu\). Moreover,
\[
\int_{X_*\times X_*} d_*(u,v)\,d\pi(u,v)
\le \int_{\overline\Omega\times\overline\Omega}|x-y|\,d\gamma(x,y).
\]
because $d_*(x,y)\le |x-y|$ on $\Omega\times\Omega$ and $d_*(x,\ast)=\delta(x)\le |x-b|$ for every $b\in\partial\Omega$. We take the infimum over $\gamma$ and obtain the inequality in question.

We now establish the reverse inequality. Let $\pi$ be any coupling of $\widetilde\mu$ and $\widetilde\nu$ on $X_*\times X_*$. We set
\[
I_\pi:=\int_{X_*\times X_*} d_*(u,v)\,d\pi(u,v).
\]
We also introduce
\[
\begin{aligned}
\pi_{II}&:=\pi\restriction_{\Omega\times\Omega},\\
\pi_{I\ast}&:=\pi\restriction_{\Omega\times\{\ast\}},\\
\pi_{\ast I}&:=\pi\restriction_{\{\ast\}\times\Omega}.
\end{aligned}
\]
Define
\[
\begin{aligned}
A&:=\{(x,y)\in\Omega\times\Omega:\ |x-y|\le \delta(x)+\delta(y)\},\\
B&:=(\Omega\times\Omega)\setminus A.
\end{aligned}
\]
For \(m\in\mathbb N\), let \(\beta_m:\Omega\to\partial\Omega\) be the Borel map from Lemma~\ref{lem:approx-boundary-selection}, and set
\[
\begin{aligned}
\alpha&:= (\pi_{II}\restriction_B)_1 + (\pi_{I\ast})_1,\\
\beta&:= (\pi_{II}\restriction_B)_2 + (\pi_{\ast I})_2.
\end{aligned}
\]
We define
\[
\begin{aligned}
\gamma_m:=\,&\pi_{II}\restriction_A
+(\mathrm{Id},\beta_m)_\#\alpha\\
&+(\beta_m,\mathrm{Id})_\#\beta.
\end{aligned}
\]
This implies that $\gamma_m\in\Adm(\mu,\nu)$. Indeed, its first marginal on $\Omega$ is
\[
(\pi_{II}\restriction_A)_1+\alpha
=(\pi_{II})_1+(\pi_{I\ast})_1
=\widetilde\mu\restriction_\Omega
=\mu,
\]
and, analogously, its second marginal on $\Omega$ equals $\nu$.

For the cost, we apply the fact that $d_*(x,y)=|x-y|$ on $A$ and $d_*(x,y)=\delta(x)+\delta(y)$ on $B$, and obtain
\begin{align*}
\int |x-y|\,d\gamma_m
&=
\int_A |x-y|\,d\pi_{II}\\
&\quad +\int_\Omega |x-\beta_m(x)|\,d\alpha(x)\\
&\quad +\int_\Omega |\beta_m(y)-y|\,d\beta(y)\\
&\le \int_A d_*(x,y)\,d\pi_{II}\\
&\quad +\Bigl(1+\frac1m\Bigr)\int_\Omega \delta(x)\,d\alpha(x)\\
&\quad +\Bigl(1+\frac1m\Bigr)\int_\Omega \delta(y)\,d\beta(y)\\
&\le \Bigl(1+\frac1m\Bigr) I_\pi.
\end{align*}
Hence
\[
W_{b,1}(\mu,\nu)\le \Bigl(1+\frac1m\Bigr) I_\pi.
\]
We first take the infimum over $\pi$ and then let $m\to\infty$ to obtain
\[
W_{b,1}(\mu,\nu)\le W_{1,d_*}(\widetilde\mu,\widetilde\nu).
\]
\end{proof}

Conceptually, the mechanism works as follows. We may encode the defect in total mass as ordinary mass placed at the collapsed boundary point \(\ast\). The boundary-reservoir geometry is then converted into an ordinary \(W_1\)-problem on the shortcut space. This reduction gives the dual formula \eqref{eq:KR-Wb1}, which is the only static feature from the shortcut representation that we need for the endpoint semigroup estimate.

\begin{proof}[Proof of Theorem~\ref{thm:KR-Wb1}]
By Lemma~\ref{lem:dstar-compact} and Proposition~\ref{prop:Wb1-augmented}, the
measures \(\widetilde\mu\) and \(\widetilde\nu\) have the same finite mass
\(m:=\mu(\Omega)+\nu(\Omega)\). If \(m=0\), the dual formula follows immediately. If
\(m>0\), we apply the standard Kantorovich--Rubinstein duality to the probability
measures \(m^{-1}\widetilde\mu\) and \(m^{-1}\widetilde\nu\) on the compact
metric space \((X_*,d_*)\), and then multiply by \(m\). This yields
\[
W_{b,1}(\mu,\nu)
=
\sup_{\Lip_{d_*}(f)\le1}\int_{X_*} f\,d(\widetilde\mu-\widetilde\nu).
\]
If we subtract the constant $f(\ast)$, we may restrict to functions that satisfy $f(\ast)=0$. If $\Lip_{d_*}(f)\le1$ and $f(\ast)=0$, then $\varphi:=f\restriction_\Omega$ satisfies
\[
|\varphi(x)|\le d_*(x,\ast)=\delta(x)\qquad\forall x\in\Omega,
\]
which implies that $\varphi$ extends continuously to $\overline\Omega$ by setting $\varphi=0$ on $\partial\Omega$. For $x,y\in\Omega$,
\[
|\varphi(x)-\varphi(y)|\le d_*(x,y)\le |x-y|,
\]
If $x\in\Omega$ and $\xi\in\partial\Omega$, then
\[
|\varphi(x)-\varphi(\xi)|=|\varphi(x)|\le \delta(x)\le |x-\xi|.
\]
Finally, \(\varphi\equiv0\) on \(\partial\Omega\). We conclude that \(\Lip(\varphi)\le1\)
on \(\overline\Omega\).

For the converse, if $\varphi\in C(\overline\Omega)$, $\varphi|_{\partial\Omega}=0$, and $\Lip(\varphi)\le1$, we define $f$ on $X_*$ by $f\restriction_\Omega=\varphi$ and $f(\ast)=0$. For every $x\in\Omega$ and $\xi\in\partial\Omega$,
\[
|\varphi(x)|=|\varphi(x)-\varphi(\xi)|\le |x-\xi|,
\]
so taking the infimum over $\xi\in\partial\Omega$ gives $|\varphi(x)|\le\delta(x)$. We then have
\[
|f(x)-f(\ast)|=|\varphi(x)|\le \delta(x)=d_*(x,\ast),
\]
and
\[
\begin{aligned}
|f(x)-f(y)|
&=|\varphi(x)-\varphi(y)|\\
&\le |x-y|,\\
|f(x)-f(y)|
&\le |f(x)-f(\ast)|+|f(\ast)-f(y)|\\
&\le \delta(x)+\delta(y),
\end{aligned}
\]
so
\[
|f(x)-f(y)|\le \min\{|x-y|,\delta(x)+\delta(y)\}=d_*(x,y),
\]
which implies $\Lip_{d_*}(f)\le1$. Finally, because $f(\ast)=0$,
\[
\int_{X_*} f\,d(\widetilde\mu-\widetilde\nu)=\int_\Omega \varphi\,d(\mu-\nu).
\]
This proves the dual formula.
\end{proof}

\section{Connected components of \texorpdfstring{$C^2$}{C2} domains}\label{app:components}
\begin{proof}[Proof of Lemma~\ref{lem:components}]
By the one-sided \(C^2\) convention we introduced in Section~\ref{sec:prelim},
\(\partial\Omega\) is a compact embedded \(C^2\) hypersurface and, for every
\(\xi\in\partial\Omega\), there exists a neighbourhood \(U_\xi\) such that
\(U_\xi\cap\Omega\) is connected. Once the boundary is flattened, \(U_\xi\cap\Omega\) is represented by a half-ball. By compactness, there exist
\(\xi_1,\dots,\xi_N\in\partial\Omega\) such that
\[
\partial\Omega\subset \bigcup_{j=1}^N U_{\xi_j}.
\]
Let \(\Omega_k\) be a connected component of \(\Omega\). Given that \(\Omega_k\) is bounded and connected, \(\partial\Omega_k\neq\emptyset\) holds. Moreover
\(\partial\Omega_k\subset\partial\Omega\). Indeed, if
\(x\in\partial\Omega_k\cap\Omega\), then \(x\notin\Omega_k\) because \(\Omega_k\) is open, so \(x\in\Omega\setminus\Omega_k\). Since \(\Omega_k\) is closed in \(\Omega\), the set \(\Omega\setminus\Omega_k\) is open in \(\Omega\). This implies that there exists \(r>0\) such that \(B_r(x)\cap\Omega\subset\Omega\setminus\Omega_k\),
which contradicts \(x\in\partial\Omega_k\). Therefore, it must be true that \(\partial\Omega_k\cap\Omega=\varnothing\), and thus \(\partial\Omega_k\subset\partial\Omega\). Let us fix \(\xi\in\partial\Omega_k\), then \(\xi\in U_{\xi_j}\) for some \(j\), and since \(\xi\in\partial\Omega_k\),
the set \(\Omega_k\) intersects \(U_{\xi_j}\cap\Omega\). As
\(U_{\xi_j}\cap\Omega\) is connected, it can intersect at most one connected
component of \(\Omega\). We conclude that \(\Omega\) has at most \(N\) connected
components.

Moreover, each \(\xi\in\partial\Omega\) belongs to the boundary of exactly one
connected component of \(\Omega\). Indeed, recall that \(U_\xi\cap\Omega\) is connected, so all interior points of \(\Omega\) sufficiently close to \(\xi\) belong to the same connected component. We now show that the closures of distinct connected
components are disjoint. To this end, suppose \(x\in\overline{\Omega_i}\cap\overline{\Omega_j}\) with \(i\neq j\). If \(x\in\Omega\), then \(x\in\overline{\Omega_j}\cap\Omega\), and since \(\Omega_j\) is closed in the relative topology of \(\Omega\), this implies \(x\in\Omega_j\). Analogously, we have \(x\in\Omega_i\), which contradicts \(\Omega_i\cap\Omega_j=\varnothing\). It follows that \(x\in\partial\Omega\). But in the previous paragraph, we showed that each point of \(\partial\Omega\) belongs to the boundary of exactly one connected component of \(\Omega\), whereas \(x\in\partial\Omega_i\cap\partial\Omega_j\). This contradiction proves that \(\overline{\Omega_i}\cap\overline{\Omega_j}=\varnothing\).

Each connected component \(\Omega_k\) is itself a bounded \(C^2\) domain. To see this, suppose that \(\xi\in\partial\Omega_k\). By the previous argument, this implies that there exists \(j\) such that \(\xi\in U_{\xi_j}\) and \(U_{\xi_j}\cap\Omega=U_{\xi_j}\cap\Omega_k\). Therefore, in the same boundary-flattening chart in which \(\partial\Omega\) is represented by a \(C^2\) graph, \(\partial\Omega_k\) is represented by that same graph near \(\xi\).

Finally, we fix \(x\in\Omega_j\). From \(\partial\Omega_j\subset\partial\Omega\), we have
\[
\dist(x,\partial\Omega)\le \dist(x,\partial\Omega_j).
\]
For the other direction, let \(y\in\partial\Omega\). If \(y\in\partial\Omega_j\), then
\[
\dist(x,\partial\Omega_j)\le |x-y|.
\]
If \(y\notin\partial\Omega_j\), then the segment that joins \(x\) to \(y\) exits
\(\Omega_j\). This means that there exists \(z\in\partial\Omega_j\) on this segment, so
\[
\dist(x,\partial\Omega_j)\le |x-z|\le |x-y|.
\]
We now take the infimum over \(y\in\partial\Omega\) and obtain
\[
\dist(x,\partial\Omega_j)\le \dist(x,\partial\Omega).
\]
The opposite inequality follows immediately from
\(\partial\Omega_j\subset\partial\Omega\). We conclude that
\[
\dist(x,\partial\Omega_j)=\dist(x,\partial\Omega)=\delta(x).
\]
\end{proof}

\section{Metric properties of \texorpdfstring{$W_{b,p}$}{Wb,p}}\label{app:Wbp-metric-proof}
The metric property of \(W_{b,p}\) also follows from the general metric-pair
framework of \cite{Che} and from the geometric realisation discussed in
\cite{BGP24}. For completeness, we include here a short direct proof to fix the normalisation that we used in the paper.

\begin{proof}[Proof of Lemma~\ref{lem:Wbp-metric}]
By Lemma~\ref{lem:adm-nonempty}, admissible plans exist for every pair $(\mu,\nu)$ and $W_{b,p}$ is finite. It remains to verify symmetry, separation, and the triangle inequality. Symmetry follows immediately. Indeed, if $\gamma\in\Adm(\mu,\nu)$ and $\tau(x,y):=(y,x)$, then $\tau_\#\gamma\in\Adm(\nu,\mu)$ and
\[
\int |x-y|^p\,d(\tau_\#\gamma)=\int |x-y|^p\,d\gamma.
\]
Moreover, $(\mathrm{Id},\mathrm{Id})_\#\mu\in\Adm(\mu,\mu)$ holds, so $W_{b,p}(\mu,\mu)=0$.

To prove separation, let us assume that $W_{b,p}(\mu,\nu)=0$. We further choose $\gamma_k\in\Adm(\mu,\nu)$ such that
\[
\int |x-y|^p\,d\gamma_k \xrightarrow[k\to\infty]{} 0.
\]
We now set
\[
\begin{multlined}
\tilde\gamma_k:=\gamma_k\restriction_{(\overline\Omega\times\overline\Omega)\setminus
(\partial\Omega\times\partial\Omega)}.
\end{multlined}
\]
It follows that $\tilde\gamma_k\in\Adm(\mu,\nu)$ and
\[
\int |x-y|^p\,d\tilde\gamma_k\le \int |x-y|^p\,d\gamma_k \xrightarrow[k\to\infty]{} 0.
\]
We fix \(\varphi\in C_c^1(\Omega)\), and let \(\bar\varphi\) denote its extension by \(0\) to \(\overline\Omega\). From $\bar\varphi=0$ on $\partial\Omega$ and the fact that $\tilde\gamma_k$ is admissible, we infer
\[
\int_\Omega \varphi\,d(\mu-\nu)
=\int_{\overline\Omega\times\overline\Omega} \bigl(\bar\varphi(x)-\bar\varphi(y)\bigr)\,d\tilde\gamma_k(x,y).
\]
We then obtain
\[
\left|\int_\Omega \varphi\,d(\mu-\nu)\right|
\le \Lip(\bar\varphi)\int |x-y|\,d\tilde\gamma_k.
\]
If $p>1$, by H\"older,
\[
\begin{aligned}
\int |x-y|\,d\tilde\gamma_k
&\le \tilde\gamma_k(\overline\Omega\times\overline\Omega)^{1-1/p}\\
&\quad \times \left(\int |x-y|^p\,d\tilde\gamma_k\right)^{1/p}.
\end{aligned}
\]
Moreover, $\tilde\gamma_k(\partial\Omega\times\partial\Omega)=0$, so
\[
\begin{aligned}
\tilde\gamma_k(\overline\Omega\times\overline\Omega)
&=\tilde\gamma_k(\Omega\times\Omega)+\tilde\gamma_k(\Omega\times\partial\Omega)\\
&\quad +\tilde\gamma_k(\partial\Omega\times\Omega)\\
&\le \mu(\Omega)+\nu(\Omega).
\end{aligned}
\]
For $p=1$, it suffices to use directly $\int |x-y|\,d\tilde\gamma_k\to0$. We conclude that $\int_\Omega \varphi\,d(\mu-\nu)=0$ for every $\varphi\in C_c^1(\Omega)$.
Since \(C_c^\infty(\Omega)\subset C_c^1(\Omega)\) is uniformly dense in \(C_c(\Omega)\), vanishing against \(C_c^1(\Omega)\) implies vanishing against all of \(C_c(\Omega)\), which describes the equality of finite Radon measures on \(\Omega\). It follows that $\mu=\nu$.

For the triangle inequality, let us fix $\gamma_{01}\in\Adm(\mu,\nu)$ and $\gamma_{12}\in\Adm(\nu,\sigma)$. Also, let
\[
\begin{aligned}
m_{01}&:=((\pi_2)_\#\gamma_{01})\restriction_{\partial\Omega},\\
m_{12}&:=((\pi_1)_\#\gamma_{12})\restriction_{\partial\Omega},
\end{aligned}
\]
and let $\Delta:\partial\Omega\to\overline\Omega\times\overline\Omega$ be $\Delta(b):=(b,b)$. We define
\[
\hat\gamma_{01}:=\gamma_{01}+\Delta_\# m_{12},\qquad
\hat\gamma_{12}:=\gamma_{12}+\Delta_\# m_{01}.
\]
The added terms are supported on $\partial\Omega\times\partial\Omega$, so
$\hat\gamma_{01}\in\Adm(\mu,\nu)$ and $\hat\gamma_{12}\in\Adm(\nu,\sigma)$.
Notice that they do not change transport costs because $|b-b|=0$ on $\partial\Omega$:
\[
\begin{aligned}
\int |x-y|^p\,d\hat\gamma_{01}&=\int |x-y|^p\,d\gamma_{01},\\
\int |x-y|^p\,d\hat\gamma_{12}&=\int |x-y|^p\,d\gamma_{12}.
\end{aligned}
\]
By construction, we have
\[
(\pi_2)_\#\hat\gamma_{01}=(\pi_1)_\#\hat\gamma_{12}
\]
as finite Borel measures on $\overline\Omega$.
We use here \(\pi_{1,2},\pi_{2,3},\pi_{1,3}:\overline\Omega^3\to\overline\Omega^2\)
to denote the coordinate projections. By the standard gluing lemma
\cite[Lemma~5.3.2]{AGS08}, there exists a finite Borel measure $\Pi$ on $\overline\Omega^3$ with
\[
\begin{aligned}
(\pi_{1,2})_\#\Pi&=\hat\gamma_{01},\\
(\pi_{2,3})_\#\Pi&=\hat\gamma_{12}.
\end{aligned}
\]
Let us set $\gamma_{02}:=(\pi_{1,3})_\#\Pi$, so $\gamma_{02}\in\Adm(\mu,\sigma)$. We apply $|x-z|\le |x-y|+|y-z|$ and Minkowski in $L^p(\Pi)$ to obtain:
\[
\begin{aligned}
\left(\int |x-z|^p\,d\gamma_{02}\right)^{1/p}
&\le \left(\int |x-y|^p\,d\hat\gamma_{01}\right)^{1/p}\\
&\quad +\left(\int |y-z|^p\,d\hat\gamma_{12}\right)^{1/p}.
\end{aligned}
\]
Therefore,
\[
\begin{aligned}
W_{b,p}(\mu,\sigma)
&\le \left(\int |x-y|^p\,d\gamma_{01}\right)^{1/p}\\
&\quad +\left(\int |y-z|^p\,d\gamma_{12}\right)^{1/p}.
\end{aligned}
\]
It suffices to take the infimum over $\gamma_{01}$ and $\gamma_{12}$ to arrive at
$W_{b,p}(\mu,\sigma)\le W_{b,p}(\mu,\nu)+W_{b,p}(\nu,\sigma)$.
\end{proof}

\section*{Data Availability}
No datasets were generated or analysed during the current study.

\section*{Conflict of interest}
The author reports no conflict of interest.

\bibliographystyle{unsrt}
\bibliography{open3}

\end{document}